\documentclass[11pt,letterpaper]{scrarticle}
\usepackage[utf8]{inputenc}
\usepackage[T1]{fontenc}
\usepackage{amsmath}
\usepackage{amsfonts}
\usepackage{amssymb}
\usepackage{amsthm}
\usepackage{mathtools}
\usepackage{euscript}
\usepackage[dvipsnames]{xcolor}

\usepackage{hyperref}
\usepackage{subcaption}
\usepackage[inner=1in,outer=1in]{geometry}
\usepackage{tikz}
\usepackage{tikz-cd}
\usetikzlibrary{positioning, calc}
\usetikzlibrary{decorations.markings}
\tikzstyle{vertex}=[inner sep=0pt]

\usepackage[backend=biber]{biblatex}
\theoremstyle{plain}
\newtheorem{thm}{Theorem}[section]
\newtheorem*{thm*}{Theorem}
\newtheorem{prop}[thm]{Proposition}
\newtheorem{lem}[thm]{Lemma}
\newtheorem{cor}[thm]{Corollary}
\newtheorem*{prop*}{Proposition}

\newtheorem{question}{Question}

\newtheorem{answer}{Answer}

\theoremstyle{definition}
\newtheorem{defn}[thm]{Definition}
\newtheorem{rmk}[thm]{Remark}

\newtheorem{nota}[thm]{Notation}

\theoremstyle{remark}
\newtheorem{example}[thm]{Example}

\numberwithin{equation}{section}

\def\on{\operatorname}
\def\sf{\mathsf}
\def\op{{\on{op}}}

\def\Mor{{\on{Mor}}}
\def\Fun{{\on{Fun}}}

\def\sDist{{\on{sDist}}}
\def\QConv{{\on{QConv}}}
\def\cint{{\on{C}\mkern-18mu\displaystyle \int  }}

\def\Cat{{\sf{Cat}}}
\def\Set{{\sf{Set}}}
\def\Vect{{\sf{Vect}}}
\def\CSet{{\sf{CSet}}}
\def\Conv{{\sf{Conv}}}
\def\Mat{{\sf{Mat}}}

\def\Assoc{{\sf{Assoc}}}
\def\Comm{{\sf{Comm}}}

\def\Prob{{\sf{FinProb}}}
\def\Fin{{\sf{Fin}}}

\def\bb{\mathbb}

\def\RR{{\bb{R}}}
\def\NN{{\bb{N}}}

\def\scr{\EuScript}
\def\I{{\scr{I}}}
\def\J{{\scr{J}}}
\def\C{{\scr{C}}}
\def\D{{\scr{D}}}
\def\P{{\scr{P}}}

\def\V{{\scr{V}}}

\definecolor{bluegray}{rgb}{0.4, 0.6, 0.8}
\definecolor{turquoise}{rgb}{0.2, 0.7, 0.6}

\usepackage{authblk}

\title{The operadic theory of convexity}
\author{Redi Haderi\footnote{redi.haderi@billent.edu.tr}}
\author{Cihan Okay\footnote{cihan.okay@bilkent.edu.tr}}
\author{Walker H. Stern\footnote{walker@walkerstern.com}}
\affil{{\small{Department of Mathematics, Bilkent University, Ankara, Turkey}}}

\begin{document}
	
	\maketitle
	
	\begin{abstract}
		In this article, we characterize convexity in terms of algebras over a PROP, and establish a tensor-product-like symmetric monoidal structure on the  category of convex sets. Using these two structures, and the theory of $\scr{O}$-monoidal categories developed in \cite{OMonGroth}, we state and prove a Grothendieck construction for lax $\scr{O}$-monoidal functors into convex sets. We apply this construction to the categorical characterization of entropy of Baez, Fritz, and Leinster, and to the study of quantum contextuality in the framework of simplicial distributions.
	\end{abstract}
	
	\tableofcontents
	
\section{Introduction}

	At its core, convexity is a deceptively simple condition. In, for example, an $\RR$-vector space $V$, the condition that a subset $C\subset V$ be convex is as simple as $C$ containing any line segment between points of $C$. However, the role that convexity plays in probability theory, optimization, and beyond quickly reveals great depths hidden behind the simple definition. 
	
	It is precisely this depth which makes abstract frameworks to study convexity desirable, and a number of such frameworks have been developed. Initially, the approach taken was simply to axiomatize the essential structure necessary to take convex combinations of elements of a set in a consistent way. This is the approach of, e.g., \cite{neumann_quasivariety_1970} and \cite{Swiriszcz}. In more recent years, a popular approach has been that of \cite{jacobs2009duality} and \cite{fritz2015convex}, which encodes an abstract version of convexity (over an arbitrary semiring $R$) in terms of algebras over a monad $D_R$ on the category $\Set$ of sets. 
	
	While the present paper starts from the perspective of convex sets as algebras over this monad, our goal is to characterize convexity in terms of operad-like structures, and to leverage the operadic perspective to provide new and useful constructions in convex sets. In particular, in Section \ref{sec:convprop}, we will introduce a PROP $\Conv_R$ whose algebras in $\Set$ are precisely $R$-convex sets in the sense of the aforementioned references, and explore its algebras in various categories. This operadic approach and perspective is presaged by a variety of works in the literature. In \cite{fritz2009presentation,fritz2015convex}, the PROP $\Conv_R$ was studied as a Lawvere theory under the name $\sf{FinStoMap}^\op$, and the operadic part of $\Conv_R$, which we term $\QConv_R$, was discussed by Leinster in \cite[Ch. 12]{Leinster_2021} and \cite{Leinster_2011} (in the former, he calls this operad $\Delta$, which notation we reserve for the simplex category, in the latter, $\mathbf{P}$).	Explicitly, the PROP $\Conv_R$ is the monoidal category with monoid of objects $\NN$, and with hom-sets $\Conv_R(m,n)$ given by the set of $n\times m$ \emph{convex $R$-valued matrices}, matrices $M$ such that the sum of the entries in each row of $M$ is $1$.  
	
	The main original contributions enabled by our operadic approach are generalizations of the Grothendieck construction (and the monoidal Grothendieck construction of \cite{MoellerVasilakopoulou}) to the convex setting. The structures that arise from this Grothendieck construction have bearing on a wide array of examples. Of particular note, the characterization of entropy given in \cite{BaezFritzLeinster} can be completely encapsulated in terms of structure-preserving functors out of a convex monoidal Grothendieck construction, as we describe in \ref{subsec:Entropy}. Similarly, generalizations of the convex monoids appearing in the study of quantum contextuality in \cite{KharoofSimplicial} arise as convex monoidal Grothendieck constructions. 
	
	The first of the two Grothendieck constructions we describe follows almost immediately from the operadic approach. By taking $\Conv_R$-algebras on both sides of the classical Grothendieck construction, we obtain
	
	\begin{prop*}
		The Grothendieck construction induces an equivalence of categories
		\[
		\begin{tikzcd}
			\cint_\scr{C}:&[-3em]\Fun(\scr{C},\CSet) \arrow[r] & \Conv(\sf{DFib}(\scr{C})).
		\end{tikzcd}
		\]
	\end{prop*}
	
	which appears in the text as Corollary \ref{cor:naiveGC}. Unwinding the definitions on the right-hand side, one sees that the appropriate notion of fibration is what we term a \emph{fibrewise convex discrete fibration}. Effectively, this simply a discrete fibration with convex structures on its fibres. This is, in itself, unremarkable, however, it is leveraging this construction that we will obtain the full, monoidal version of the Grothendieck construction. 
	
	\subsection{The convex tensor product}
	
	Our first substantial task in this paper is to provide explicit, computationally tractable constructions in convex sets. In particular, we characterize coequalizers of convex sets as set-theoretic quotients by equivalence relations which respect the convex structure. We leverage this description to give an explicit characterization of the coproduct of convex sets, which, motivated by the most geometric examples, we call the \emph{join}. 
	
	However, the construction of the greatest import is that of the \emph{convex tensor product}, which provides the symmetric monoidal structure at the heart of our Grothendieck construction. A convex analogue of bilinearity, which we term \emph{biconvexity}, shows up naturally in a number of settings. Formally, given convex sets $X$, $Y$, and $Z$, a map 
	\[
	\begin{tikzcd}
		f:&[-3em] X\times Y \arrow[r] &Z
	\end{tikzcd}
	\]
	is biconvex when, for any convex combinations $\sum_i \alpha_i x_i$ and $\sum_j \beta_j y_j$ in $X$ and $Y$, respectively, we have 
	\[
	f\left(\sum_i\alpha_i x_i,\sum_j \beta_j y_j\right)=\sum_{i,j} \alpha_i\beta_j f(x_i,y_j). 
	\]
	Biconvexity appears in a variety of contexts, from the composition maps when one attempts to enrich convex sets over themselves, to the convex categories and monoids studied in \cite{KharoofSimplicial}. 
	
	In precise analogy to the tensor product of vector spaces, we show that there is a symmetric monoidal structure $\otimes$ on the category $\CSet$ of convex sets, and that $X\otimes Y$ corepresents biconvex maps out of $X\times Y$. We also prove that the aforementioned convex categories of \cite{KharoofSimplicial} are precisely categories enriched over $(\CSet,\otimes)$. 
	
	\subsection{$\scr{O}$-monoidal categories and the Grothendieck construction}
	
	The second ingredient in our Grothendieck construction is a promotion of the classical Grothendieck construction to take as input a wide variety of monoidal structures. More precisely, corresponding to any operad $\scr{O}$ in $\Set$, we define a corresponding notion of $\scr{O}$-\emph{monoidal category}, which has monoidal-product-like $n$-ary operations for each operation in $\scr{O}(n)$, satisfying some coherence conditions. 
	
	In more detail: in a symmetric monoidal category $(\C,\otimes)$, one can define $n$-fold tensor product functors 
	\[
	\begin{tikzcd}
		\displaystyle\bigotimes_n :&[-3em] \C^n \arrow[r] & \C 
	\end{tikzcd}
	\]
	as well as natural isomorphisms between them and their composites, which must satisfy additional conditions. Generalizing this, in, for example, a $\QConv_R$-monoidal category $(\C,\otimes)$, every convex vector $(\alpha_1,\ldots,\alpha_n)$ in $R^n$ gives rise to an operation 
	\[
	\begin{tikzcd}
		\displaystyle\bigotimes_{(\alpha_1,\ldots,\alpha_n)} :&[-3em] \C^n \arrow[r] & \C 
	\end{tikzcd}
	\]
	and these operations must be compatible with the operadic composition up to natural isomorphism. There are then corresponding notions of lax $\scr{O}$-monoidal functors, and $\scr{O}$-monoidal natural transformations, which generalize the corresponding notions for symmetric monoidal categories.
	
	In a separate paper \cite{OMonGroth}, we develop a $\Set$-valued $\scr{O}$-monoidal Grothendieck construction, which we then leverage here to prove
	
	\begin{thm*}
		For a $\scr{O}$-monoidal category $(\scr{I},\odot)$, the classical Grothendieck construction induces an equivalence of categories 
		\[
		\begin{tikzcd}
			\cint_{\scr{I}}^\scr{O}:&[-3em] \Fun^{\scr{O},\on{lax}}((\scr{I},\odot),(\CSet,\otimes))\arrow[r] & \scr{O}\sf{CFib}_{\scr{I}}.
		\end{tikzcd}
		\]
	\end{thm*} 
	
	Unwinding the notation, this says that lax $\O$-monoidal functors from $(\scr{I},\odot)$ to $(\CSet,\otimes)$ are the same thing as fibrewise convex discrete fibrations over $(\scr{I},\odot)$ which are \emph{also} strict $\O$-monoidal functors.

	Owing to the generality and abstraction necessary to state and prove our Grothendieck construction, it may seem that our design in this paper is to lure the reader in with tangible and concrete computations, and then suddenly bash them over the head with abstract categorical machinery. In fact, nothing could be further from the truth. We hope that by the end of the present paper, the reader will see that the categorical machinery is itself, in its application to convexity, concrete, tangible, and comprehensible.
	
	\subsection{Structure of the paper}
	
	In section \ref{sec:Csets}, we briefly review the monadic definition of $R$-convex sets and the category $\CSet_R$ of $R$-convex sets. We then define convex relations, and characterize colimits in $\CSet_R$ in terms of them. Finally, we define the convex tensor product, and establish many of its properties. In Section \ref{sec:convprop}, we introduce the convexity PROP and study its structure and algebras. In section \ref{sec:GC}, we establish the basic definitions and properties of $\scr{O}$-monoidal categories. We then state and prove the two convex Grothendieck constructions described above. The final section, section \ref{sec:EGs} describes applications of the convex Grothendieck construction to entropy and quantum contextuality. 
	
	\subsection{Acknowledgements}
	This work is supported by the US Air Force Office of Scientific Research under award
	number FA9550-21-1-0002. The second author acknowledges support from the Digital Horizon Europe project FoQaCiA, GA no. 101070558.
	
\section{Convex sets}\label{sec:Csets}

Before we study convex objects in greater generality, we first recall the existing theory of convex sets, and introduce some explicit constructions to allow us to better work with convex sets categorically. 

\subsection{The convexity monad and first properties}

The classic categorical approach to convex sets, as exemplified by \cite{Swiriszcz,fritz2015convex,jacobs2009duality}, is to consider sets structured over a monad of distributions. While we will eventually move towards encoding convexity in term of algebras over PROPs and operads, we begin from this prior approach. 

\begin{defn}
	By a \emph{semiring} we will always mean a commutative semiring, that is, a set $R$ together with unital, commutative monoid structures $+$ and $\cdot$ (with units $0$ and $1$ respectively) such that $\cdot$ distributes over $+$, and $0\cdot r=0$ for any $r\in R$. 
	
	We will call a semiring $R$ a \emph{semifield} if every element other than $0$ admits a multiplicative inverse. 
\end{defn}

\begin{example}\
	\begin{enumerate}
		\item The prototypical example of a semiring, and the one which will guide us throughout this paper, is the set $\RR_{\geq 0}$ of non-negative real numbers, equipped with the usual addition and multiplication of real numbers. This is, in particular, a semifield. 
		\item The \emph{Boolean semiring} is the set $\{0,1\}$ viewed as truth values, equipped with logical `and' and `or'. 
		\item Any ring is, in particular a semiring.  
	\end{enumerate}
\end{example}

\begin{defn}
	Let $R$ be a semiring. The \emph{convexity monad} (or, as it is called elsewhere in the literature, the \emph{distributions monad}) is the functor 
	\[
	\begin{tikzcd}
		D_R: &[-3em] \Set \arrow[r] & \Set 
	\end{tikzcd}
	\]
	which sends $X$ to the set 
	\[
	D_R(X):=\left\lbrace p:X\to R \middle| p \text{ fin. supp. }, \sum_{x\in X} p(x)=1\right\rbrace. 
	\]
	and a function $f:X\to Y$ to the map $f_\ast$  which sends $p$ to 
	\[
		f_\ast(p)(y)=\sum_{x\in f^{-1}(y)}p(x).
	\]
\end{defn}

To properly consider $D_R$ as a monad, we must define structure maps displaying it as an algebra in $\Fun(\Set,\Set)$. The first of these is the multiplication 
\[
\begin{tikzcd}
	\mu_X: &[-3em] D_R(D_R(X))\arrow[r] & D_R(X) 
\end{tikzcd}
\]
Which sends a distribution $p$ on $D_R(X)$ to the distribution $\mu_X(p)$ given by 
\[
\mu(p)(x)=\sum_{q\in D_R(X)} p(q)q(x). 
\]
The unit is the transformation with components 
\[
\begin{tikzcd}
	\delta_X: &[-3em] X \arrow[r] & D_R(X)
\end{tikzcd}
\]
defined by $\delta_X(x)(y)=\delta_{x,y}$, where the latter $\delta$ denotes the Dirac delta. 

\begin{rmk}
	The monad $D_R$ appears, under various names, in \cite{fritz2015convex} and \cite{jacobs2009duality}. We mostly follow the notational conventions from \cite{OkaySimplicial} and \cite{KharoofSimplicial}, which hew closely to those of \cite{jacobs2009duality}.
\end{rmk}

\begin{defn}
	An $R$\emph{-convex set} is an algebra over the monad $D_R$. The category of $R$-convex sets will be denote $\CSet_R$. In the special case where $R=\RR_{\geq 0}$, we will simply write $\CSet$ for $\CSet_{\RR_{\geq 0}}$. Similarly, we will simply write $D$ for $D_{\RR_{\geq 0}}$. 
\end{defn}

We will sometimes write a formal convex combination $p\in D_R(X)$ by 
\[
\sum_{x\in \on{Supp}(p)} p(x)\bullet x 
\]
for ease of readability. If $X$ is a convex set with structure map $\pi^X$, we will write 
\[
\sum_{x\in X} p(x)x:=\pi^X\left(\sum_{x\in X} p(x)\bullet x \right)
\]
for the corresponding convex combination in $X$. We will also sometimes allow sums with repeated instances of the same element, so that, for instance, for $x\in X$ and $\alpha_i\in \RR_{\geq 0}$,
\[
\sum_{i} \alpha_i x:= \left(\sum_i \alpha_i\right) x.
\]

\begin{nota}
	For ease of notation, we will write $\Delta^n_R:=D_R(\{0,1,2,\ldots,n\})$, and simply $\Delta^n$ when $R=\RR_{\geq 0}$. 
\end{nota}

The following will be of use in explicit constructions of convex sets as quotients. 

\begin{defn} \label{def : convrel}
	Let $\pi^X:D_R(X)\to X$ be an $R$-convex set. A \emph{convex relation} on $X$ is a relation $\sim$ on $X$ such that, for any $\alpha\in \Delta^n_R$ if $x_i\sim y_i$ for $i\in [n]$, then 
	\[
	\sum_{i\in [n]} \alpha_i x_i \sim \sum_{i\in [n]} \alpha_i y_i. 
	\]
	A \emph{convex equivalence relation} is a convex relation which is also an equivalence relation. 
	
	Given an equivalence relation (not necessarily convex) $\sim$ on $X$, we define an equivalence relation $\sim_D$ on $D_R(X)$ by declaring $p\sim_D q$ if and only if, for every equivalence class $[x]\in X_{/\sim}$, 
	\[
	\sum_{y\in [x]} p(y)=\sum_{y\in [x]} q(y).
	\]
\end{defn}

\begin{lem}
	An equivalence relation $\sim$ on $X$ is a convex equivalence relation if and only if, for any $p\sim_D q$ $\pi^X(p)\sim \pi^X(q)$. 
\end{lem}

\begin{proof}
	If $p\sim_D q$, we can define a new distribution $r$ by choosing a set $\{x_i\}$ of representatives for the equivalence classes of $X_{/\sim}$, and defining 
	\[
	r(z):=\begin{cases}
		\sum_{y\in [z]} p(y) & \exists i\text{ s.t. } z=x_i\\
		0 & \text{else}.
	\end{cases}
	\]
	Then, by construction $p\sim_D r\sim_D q$. 
	
	If $\sim$ is a convex equivalence relation, then 
	\[
	\sum_{x\in X} p(x)x\sim \sum_{[x_i]\in X_{/\sim}}\sum_{y\in [x_i]} p(x)x_i =\sum_{[x_i]\in X_{/\sim}} r(x)x_i =\sum_{x\in X} r(x)x
	\]
	so $\pi^X(p)\sim \pi^X(r)$, and similarly for $q$. Thus, $\pi^X(p)\sim \pi^X(q)$. 
	
	On the other hand, suppose $p\sim_D q$ implies that $\pi^X(p)\sim \pi^X(q)$. Let $\alpha\in \Delta^n_R$ and $x_i\sim y_i$. Define 
	\[
	p(x)= \sum_{j\; :\;x_j=x} p(x_j)\quad \text{and} \quad q(x)= \sum_{j\; :\;y_j=x} p(y_j)
	\]
	Then, by definition
	\[
	\pi^X(p)= \sum_{i} \alpha_i x_i\quad \text{and} \quad \pi^X(q)= \sum_{i} \alpha_i y_i
	\]
	and $p\sim_D q$. Thus,  
	\[
	\sum_{i} \alpha_i x_i \sim \sum_{i} \alpha_i y_i
	\]
	as desired.
\end{proof}

\begin{lem}
	Let $\sim$ be a convex relation. Then the equivalence relation generated by $\sim$ is a convex equivalence relation. 
\end{lem}

\begin{proof}
	Follows from the construction of the equivalence relation generated by $\sim$ as zig-zags. Given $\alpha\in \Delta^n$ and zig-zags from $x_i$ to $y_i$, one can standardize the lengths of these zig-zags by including identities, and then apply the convexity stepwise. 
\end{proof}

\begin{lem}
	Let $\pi^X:D_R(X)\to X$ be an $R$-convex set and $\sim$ a convex equivalence relation on $X$. Write $Q:X\to X_{/\sim}$ for the quotient map. Define
	\[
	\begin{tikzcd}[row sep=0em]
		\pi^\sim: &[-3em] D_R(X_{/\sim}) \arrow[r] & X_{/\sim}\\
		& P \arrow[r,mapsto] & {[\pi^X(\overline{P})] }
	\end{tikzcd}
	\]
	where $\overline{P}\in D_R(X)$ is any distribution with $D_R(Q)(\overline{P})=P$. Then $\pi^\sim$ is well-defined and 
	equips $X_{/\sim}$ with the structure of a convex set. 
\end{lem}

\begin{proof}
	We first show that the map $\pi^\sim$ is well-defined. Suppose $\overline{P}$ and $\widetilde{P}$ are two distributions satisfying $D_R(Q)(\overline{P})=D_R(Q)(\widetilde{P})=P$. This means that for each equivalence class $[x]\in X_{/\sim}$, we have that 
	\[
	\sum_{y\in [x]} \overline{P}(y)= \sum_{y\in [x]} \widetilde{P}(y).
	\]
	So that $\overline{P}\sim_D \widetilde{P}$. Thus $\pi^X(\overline{P})\sim \pi^X(\widetilde{P})$, and so the map is well-defined. 
	
	We first check unitality. Given $[x]\in X_{/\sim}$, $D_R(Q)(\delta_x)=\delta_{[x]}$, and thus 
	\[
	\pi^\sim(\delta_{[x]})=[\pi^X(\delta_x)]=[x]
	\]
	as desired. 
	
	On the other hand, for associativity, consider the diagram
	\[
	\begin{tikzcd}
		& D^2_R(X)\arrow[rrr,"D_R(\pi^X)"]\arrow[ddd,"\mu"] & & & D_R(X)\arrow[ddd,"\pi^X"] \\
		D^2_R(X_{/\sim})\arrow[rrr,"D_R(\pi^\sim)"]\arrow[ddd,"\mu"]\arrow[ur,leftarrow,"D^2_R(Q)"] & & & D_R(X_{/\sim})\arrow[ddd,"\pi^\sim"]\arrow[ur,leftarrow,"D_R(Q)"] & \\
		& & & &\\
		& D_R(X)\arrow[rrr,"\pi^X"] & & & X\\
		D_R(X_{/\sim})\arrow[rrr,"\pi^\sim"']\arrow[ur,leftarrow,"D_R(Q)"] & &  & X_{/\sim}\arrow[ur,leftarrow,"Q"] & 
	\end{tikzcd}
	\]
	every face except possibly the front fact commutes by construction. However, the maps $Q$, $D_R(Q)$ and $D^2_R(Q)$ are all surjective, and so the front face commutes as well. 
\end{proof}

\begin{defn}
	For an $R$-convex set $X$ and a relation $\sim$ on $X$, the \emph{convex closure} of $\sim$ is the minimal convex equivalence relation on $X$ containing $\sim$. 
\end{defn}

\begin{lem}\label{lem:conv_closure_enough}
	Let $X,Y\in \CSet_R$, and let $\sim$ be a relation on $X$. Denote by $\simeq$ the convex closure of $\sim$. Then a convex function $f:X\to Y$ descends to a convex function $f:X_{/\simeq}\to Y$ if and only if $f(x)=f(y)$ for all $x\sim y$.
\end{lem}

\begin{proof}
	This follows by explicit construction. We denote by $\sim_C$ the relation on $X$ defined by $x\sim_C y$ if and only if there exist $\alpha\in \Delta^n_R$ and $x_i\sim y_i$ in $X$ such that 
	\[
	\sum_i \alpha_i x_i =x \quad \text{and}\quad \sum_i \alpha_i y_i =y.
	\]
	It is immediate that this relation is a convex relation, and a convex function $f:X\to Y$ sends $x\sim_C y$ to $f(x)=f(y)$ if and only if $f(x)=f(y)$ for all $x\sim y$. We can then take the equivalence relation generated by $\sim_C$. Since any convex equivalence relation which contains $\sim$ must contain this convex equivalence relation, it follows that this is $\simeq$. The lemma is then proven, since any $f:X\to Y$ which sends $\sim_C$ to identities must also send $\simeq$ to identities. 
\end{proof}

\begin{cor}
	Let $f,g:X\to Y$ be convex maps between convex sets. The coequalizer of $f$ and $g$ in $\CSet_R$ is the quotient of $Y$ by the convex closure of the relation $f(x)\sim g(x)$. 
\end{cor}

\subsection{The join}

Note that since $\CSet_R$ is a category of algebras over a monad on $\Set$, it is complete and cocomplete (\cite{appelgate1969coequalizers}). The purpose of this section is to construct coproducts explicitly. Thoughout, we require $R$ to be a \emph{semifield}, as we will make liberal use of inverses.  

\begin{defn}
	Let $X_1,X_2\in \CSet_R$. The \emph{join} of $X_1$ with $X_2$ is the quotient of $D_R(X_1\amalg X_2)$ by the equivalence relation generated by requiring  
	\[
	\begin{tikzcd}
		p: &[-3em] X_1\amalg X_2 \arrow[r] & {[0,1]} 
	\end{tikzcd}
	\] 
	to be equivalent to the function $\overline{p}$ which sends 
	\[
	\left(\sum_{x\in X_i}\frac{p(x)}{\sum_{x\in X_i} p(x)}x\right)
	\longmapsto 
	\left(\sum_{x\in X_i} p(x)\right) 
	\]
	for $i\in \{1,2\}$ when the latter sum is non-zero, and otherwise takes value $0$. We denote the join of $X_1$ and $X_2$ by $X_1\star X_2$. 
\end{defn}

\begin{rmk}
	We can rewrite the defining equivalence relation to read $p\sim q$ if and only if 
	\[
	\sum_{x\in X_i} p(x)=\sum_{x\in X_i} q(x) 
	\]
	and 
	\[
	\sum_{x\in X_i}\frac{p(x)}{\sum_{x\in X_i} p(x)}x=\sum_{x\in X_i}\frac{q(x)}{\sum_{x\in X_i} q(x)}x
	\]
	for each $i\in \{1,2\}$ such that the corresponding sum of probabilities is non-zero, and such that if 
	\[
	\sum_{x\in X_i} p(x)=0
	\]
	then so also must
	\[
	\sum_{x\in X_i} q(x)=0.
	\]
	The terminology  ``join" is motivated by this observation, since, in the $\RR_{\geq 0}$-coefficient case, it tells us we can uniquely identify elements of $X_1\star X_2$ with points on line segments between points of $X$ and points of $Y$. More explicitly, we identify elements of $X\star Y$ with equivalence classes
	\[
	[\alpha,x,y]\in X\coprod_{X\times Y\times \{0\}} X\times Y\times \Delta^1_R \coprod_{X\times Y\times \{1\}} Y
	\]
	by associating such triple with the equivalence class of 
	\[
	p(z)= \begin{cases}
		\alpha_0 & z=x \\
		\alpha_1 & z=y \\
		0 & \text{else}.
	\end{cases}
	\]
\end{rmk}

\begin{lem}
	The join of $X,Y\in \CSet_R$ is a convex set.
\end{lem}

\begin{proof}
	We show that the equivalence relation $\sim$ defining the join is a convex equivalence relation.  If $\alpha\in \Delta^n_R$ and $p_j\sim q_j$ for all $j\in [n]$, then 
	\[
	\begin{aligned}
		\sum_{x\in X_i} \sum_j \alpha_j p_j(x) & =\sum_j \alpha_j \sum_{x\in X_i}  p_j(x) \\
		&=\sum_j \alpha_j \sum_{x\in X_i}  q_j(x)\\
		&=\sum_{x\in X_i} \sum_j \alpha_j q_j(x).
	\end{aligned}
	\]
	A similar computation shows 
	\[
	\sum_{x\in X_i}\frac{\sum_j \alpha_j p_j(x)}{\sum_{x\in X_i} \sum_j\alpha_j p_j(x)}x=\sum_{x\in X_i}\frac{\sum_j \alpha_j q_j(x)}{\sum_{x\in X_i} \sum_j\alpha_j q_j(x)}x
	\]
	when the denominators are non-zero. Thus,
	\[
	\sum_j\alpha_jp_j\sim \sum_j\alpha_jq_j
	\]
	as desired. 
\end{proof}

\begin{rmk}
	We can compute the structure map $\pi:D(X\star Y)\to X\star Y$ explicitly, using the representation of $X\star Y$ in terms of equivalence classes of triples. In these terms, the structure map is
	\[
	\begin{tikzcd}[row sep=0em]
		\pi:&[-3em] D_R(X\star Y) \arrow[r] & X\star Y \\
		&\sum_j \beta_j {[\alpha_j,x_j,y_j]} \arrow[r,mapsto] & \left[\sum_j\beta_j\alpha_j,\sum_j\beta_jx_j,\sum_j\beta_jy_j\right]
	\end{tikzcd}
	\]
\end{rmk}

\begin{lem}
	The join $X\star Y$ is the coproduct in the category of convex sets. 
\end{lem}

\begin{proof}
	We note that there are canonical convex inclusions $i_X:X\to X\star Y$ and $i_Y\to X\star Y$ given by sending, $x\mapsto [1,x,y]$ and $y\mapsto [0,x,y]$. Given convex maps $f:X\to Z$ and $g:Y\to Z$, we can define a convex map $f\star g:X\star Y\to Z$ by $[\alpha,x,y]\mapsto \alpha f(x)+(1-\alpha)g(y)$. The assignment $(f,g)\mapsto f\star g$ is clearly inverse to the assignment $h\mapsto (h\circ i_X,h\circ i_Y)$, and so $X\star Y$ is the coproduct. 
\end{proof}

\begin{cor}
	The join defines a symmetric monoidal structure on $\CSet_R$ with unit $\varnothing$. 
\end{cor}

\begin{rmk}
	In general, given an indexing set $J$ and convex sets $\{X_j\}_{j\in I}$, the join $\bigstar_{j\in J} X_j$ can be defined to be the set of equivalence classes in 
	\[
	D_R(J)\times \prod_{j\in J}X_j 
	\] 
	under the relation that $(p,(x_j))\sim (p,(y_j))$ when $x_j=y_j$ for all $j\in J$ with $p(j)\neq 0$. 
\end{rmk}

\subsection{The convex tensor product}

The definition of the tensor product is motivated by the natural inner hom of convex sets. We remain in the setting where $R$ is a semifield. 

\begin{defn}
	Given $X,Y\in \CSet_R$, we define a convex set $Y^X$ with underlying set $\CSet_R(X,Y)$ and convex structure defined by 
	\[
	\pi(p)(x) =\sum_{f\in \on{Supp}(p)} p(f) f(x). 
	\] 
\end{defn}

When we try to construct a convex version of the usual tensor-hom adjunction, we would want an isomorphism of the form
\[
\begin{tikzcd}[row sep=0em]
	\CSet_R(X\times Y,Z) \arrow[r,phantom,"?"{above=1em}]\arrow[r,phantom,"\cong"{description}] & \CSet_R(Y,Z^X)\\
	f \arrow[r,mapsto] & (y \mapsto f(-,y)) \\
	((x,y)\mapsto g_y(x)) & (y\mapsto g_y).\arrow[l,mapsto]
\end{tikzcd}
\]
Unfortunately, these maps do not define such a bijection. We notice that $f\in \Set(X\times Y, Z)$ corresponds to some $g\in \CSet_R(Y,Z^X)$ if and only if it satisfies the condition 
\[
f(\sum_i\alpha_i x_i,\sum_j \beta_j y_j)=\sum_{i,j} \alpha_i\beta_j f(x_i,y_j).
\]
This condition looks, superficially, analogous to bilinearity, and so we aim to construct a tensor product of convex sets corresponding to this property. 

\begin{defn}
	For $n\in \NN$, denote by $\underline{n}$ the set $\{1,\ldots,n\}$. Let $X_i\in \CSet_R$ for $i\in \underline{n}$, and let $Z\in \CSet_R$. We call a map 
	\[
	\begin{tikzcd}
		f:&[-3em] \prod_{i\in \underline{n}} X_i \arrow[r] & Z 
	\end{tikzcd}
	\] 
	\emph{$n$-convex} if, for every $i\in \underline{n}$, every $\alpha\in \Delta^k_R$, every $(k+1)$-tuple $x_j\in X_i^{(k+1)}$ 
	we have 
	\[
	f\left(y_1,\ldots,y_{i-1},\sum_{j} \alpha_jx_j,y_{i+1},\ldots, y_n\right)=\sum_{j} \alpha_j f\left(y_1,\ldots,y_{i-1},x_j,y_{i+1},\ldots, y_n\right)
	\]
	In particular, we call $f:X\times Y \to Z$ ($R$-)\emph{biconvex} if
	\[
	f\left(\sum_i\alpha_i x_i,\sum_j \beta_j y_j\right)=\sum_{i,j} \alpha_i\beta_j f(x_i,y_j)
	\]
	for any convex combinations $\sum_i\alpha_ix_i$ and $\sum_j \beta_j y_j$. We will denote the set of $n$-convex maps $\prod_i X_i\to Z$ by $\Conv_n((X_i)_{i\in\underline{n}},Z)$.
\end{defn}

\begin{lem}\label{lem:Conv_funct_in_Z}
	The composition of an $n$-convex map $\prod_{i\in \underline{n}}X_i\to Z$ with a convex map $Z\to W$ is $n$-convex, in particular, $\Conv_n$ defines a functor $\on{Conv}_n(\{X_i\}_{i\in \underline{n}},-): \CSet_R \to \Set$.  
\end{lem}

\begin{proof}
	Immediate from unwinding the definitions.
\end{proof}

\begin{lem}\label{lem:Conv_funct_in_Conv}
	Given a surjective map of sets $\phi:\underline{n}\to \underline{m}$, $\underline{m}$-indexed collection $X_j$ of convex sets, and an $\underline{n}$-indexed collection $Y_i$ of convex sets, every collection 
	\[
	\{\prod_{i\in \phi^{-1}(j)} Y_i\to X_j\}_{j\in \underline{m}}
	\]
	of $|\phi^{-1}(j)|$-convex maps induces a unique natural map  
	\[
	\Conv_m\left((X_j)_{j\in \underline{m}},Z\right)\longrightarrow \Conv_{n} \left((Y_i)_{i\in \underline{n}},Z\right)
	\]
	by composition.
\end{lem}

\begin{proof}
	Immediate from unwinding the definitions. 
\end{proof}

\begin{rmk}
	We can rephrase our discussion of the universal property of $Z^X$ by saying that there is a bijection 
	\[
	\Conv_2((X,Y),Z)\cong \CSet_R(Y,Z^X).
	\]
	natural in $X$, $Y$, and $Z$. 
\end{rmk}

\begin{defn}
	Given an $n$-indexed set of convex sets $\{X_i\}_{i\in \underline{n}}$ define the \emph{$n$-indexed convex tensor product} of the $X_i$ to be the quotient of 
	\[
	\bigotimes_{i\in\underline{n}} X_i :=D_R\left(\prod_{i\in\underline{n}} X_i\right)_{/\sim}
	\]
	by the equivalence relation generated by 
	\begin{equation}\label{eq:genTprodReln}
		1\bullet\left(\sum_{j_i} \alpha_{j_i}^i \cdot x_{j_i}^i \right)_{i\in \underline{n}}\sim \sum_{(j_1,\ldots,j_n)} \alpha_{j_1}^1\cdots \alpha_{j_n}^n \bullet (x_{j_1}^1,\ldots, x_{j_n}^n).
	\end{equation}
\end{defn}

\begin{prop}\label{prop:univ_prop_TP}
	The $\underline{n}$-indexed tensor product represents $\Conv_n$. That is, there is a natural isomorphism 
	\[
	\Conv_n((X_i)_{i\in \underline{n}}, Z)\cong \CSet_R\left(\bigotimes_{i\in \underline{n}} X_i ,Z\right).
	\]
\end{prop}

\begin{proof}
	By the universal property of the free convex set, there is a bijective correspondence 
	\[
	\Set\left(\prod_{i\in\underline{n}} X_i,Z \right) \cong \CSet_R\left(D_R\left(\prod_{i\in\underline{n}} X_i\right),Z\right)
	\]
	Which sends a map of sets $f$ to the map $\overline{f}$ given by 
	\[
	\overline{f}\left(\sum_{i} \alpha_i (x^1_i,\ldots, x^n_i) \right)=\sum_i \alpha_i f(x_i^1,\ldots,x_i^n).
	\]
	Moreover, by Lemma \ref{lem:conv_closure_enough}, there is a natural isomorphism 
	\[
	\CSet^{\on{gen}}_R\left(D_R\left(\prod_{i\in\underline{n}} X_i\right),Z\right)\cong \CSet_R\left(\bigotimes_{i\in\underline{n}} X_i,Z\right)
	\]
	where the superscript $\on{gen}$ denotes those maps $\overline{f}$ which send the generating relations of Eq. (\ref{eq:genTprodReln}) to identities. However, under our first natural bijection, such maps correspond precisely to maps of sets $f$ such that 
	\[
	f\left(\left(\sum_{j_i}\alpha_{j_i}^ix_i\right)_{i\in \underline{n}}\right) =\sum_{(j_1,\ldots,j_n)} \alpha_{j_1}^1\cdots \alpha_{j_n}^n f(x_{j_1}^1,\ldots, x_{j_n}^n)
	\]
	that is, precisely the $n$-convex maps. 
\end{proof}

\begin{rmk}
	Unlike the defining relation for the join, the defining relation for the tensor product is rather inexplicit, and thus ill-suited for direct computation. This is, however, entirely analogous to the case of algebraic tensor products, and the same workaround is effective in this case. Every element of $X\otimes Y$ can be expressed as a convex combination of the \emph{pure tensors} $x\otimes y=[(x,y)]$ for $x\in X$ and $y\in Y$. As such, one may formally manipulate convex combinations of pure tensors according to the rule that 
	\[
	\left(\sum_i\alpha_i x_i\right)\otimes \left(\sum_j\beta_j y_j\right)=\sum_{i,j} \alpha_i\beta_j x_i\otimes y_j 
	\]
	to compute effectively. Note, too, that although it may not seem immediate from the definition, Proposition \ref{prop:univ_prop_TP} shows that the tensor product of two convex sets is almost never a singleton. 
\end{rmk}

The characterization of Proposition \ref{prop:univ_prop_TP} has a long list of corollaries.

\begin{cor}\label{cor:univ_map_TP}
	The map  
	\[
	\begin{tikzcd}[row sep=0em]
		U_{(X_i)}:&[-3em] \prod_{i\in\underline{n}} X_i\arrow[r]&  \bigotimes_{i\in\underline{n}} X_i \\
		& (x_1,\ldots,x_n) \arrow[r,mapsto] & x_1\otimes \cdots\otimes x_n 
	\end{tikzcd}
	\]
	is $n$-convex. Restriction along $U_{(X_i)}$ yields the isomorphism of Proposition \ref{prop:univ_prop_TP}. 
\end{cor}

\begin{proof}
	Unwinding the image of the identity map under the isomorphism of Proposition \ref{prop:univ_prop_TP}
\end{proof}

\begin{cor}\label{cor:Set_maps_induce_tp_maps}
	Given a surjective map of sets $\phi:\underline{n}\to \underline{m}$, $\underline{m}$-indexed collection $X_j$ of convex sets, and an $\underline{n}$-indexed collection $Y_i$ of convex sets, every collection 
	\[
	\{f_j:\prod_{i\in \phi^{-1}(j)} Y_i\to X_j\}_{j\in \underline{m}}
	\]
	of $|\phi^{-1}(j)|$-convex maps induces a unique map 
	\[
	\bigotimes_{i\in \underline{n}} Y_i \to \bigotimes_{j\in \underline{m}} X_j
	\]
	such that the diagram 
	\[
	\begin{tikzcd}
		\prod_{i\in \underline{n}} Y_i \arrow[r]\arrow[d,"U_{(Y_i)}"'] & \prod_{j\in \underline{m}} X_j\arrow[d,"U_{(X_j)}"]\\
		\bigotimes_{i\in \underline{n}} Y_i \arrow[r] & \bigotimes_{j\in \underline{m}} X_j
	\end{tikzcd}
	\]
	commutes.
\end{cor}

\begin{proof}
	Combine Lemma \ref{lem:Conv_funct_in_Conv} with Corollary \ref{cor:univ_map_TP}. 
\end{proof}

\begin{cor}
	The assignment 
	\[
	\begin{tikzcd}[row sep=0em]
		-\otimes -: &[-3em] \CSet_R\times \CSet_R \arrow[r] & \CSet_R \\
		& (X,Y) \arrow[r,mapsto] & X\otimes Y
	\end{tikzcd}
	\]
	defines a functor. More generally, the $n$-fold tensor product defines a functor 
	\[
	\begin{tikzcd}
		\CSet^{\times n}_R\arrow[r] & \CSet_R. 
	\end{tikzcd}
	\]
\end{cor}

\begin{proof}
	The first statement follows by restricting Corollary \ref{cor:Set_maps_induce_tp_maps} to pairs of convex maps, the second by restricting to tuples.
\end{proof}
%
%

\begin{defn}
	Denote a chosen singleton set by $\mathbf{1}$, and abusively use the same notation for the unique convex set whose underlying set is $\mathbf{1}$. We fix notation for the coherences of the cartesian monoidal structure on $\Set$. 
	\begin{itemize}
		\item For sets $A,B,C\in \Set$, we write
		\[
		\begin{tikzcd}
			\alpha_{A,B,C}: &[-3em] A\times (B\times C) \arrow[r,"\cong"] & (A\times B)\times C 
		\end{tikzcd}
		\]
		for the associators (rebracketing isomorphisms). 
		\item For $A\in \Set$, we write 
		\[
		\begin{tikzcd}
			\lambda_A: &[-3em] \mathbf{1}\times A \arrow[r,"\cong"] & A 
		\end{tikzcd}
		\]
		and 
		\[
		\begin{tikzcd}
			\rho_A: &[-3em] A\times \mathbf{1} \arrow[r,"\cong"] & A 
		\end{tikzcd}
		\]
		for the left and right unitors. 
		\item For $A,B\in \Set$, we write 
		\[
		\begin{tikzcd}
			\sigma_{A,B}: &[-3em] A\times B\arrow[r,"\cong"] & B\times A
		\end{tikzcd}
		\]
		for the braiding. 
	\end{itemize}	
	By Corollary \ref{cor:Set_maps_induce_tp_maps}, these morphisms induce \emph{unique} maps $a_{A,B,C}$, $\ell_{A}$, $r_{A}$, and $s_{A,B}$ for any $A,B,C\in \CSet_R$ such that the following diagrams commute, where the vertical morphisms are appropriate composites of $U$'s.
	\begin{itemize}
		\item The diagrams 
		\[
		\begin{tikzcd}
			A\times(B\times C) \arrow[r,"{\alpha_{A,B,C}}"]\arrow[d] & (A\times B) \times C\arrow[d]\\
			A\otimes(B\otimes C) \arrow[r,"{a_{A,B,C}}"'] & (A\otimes B) \otimes C
		\end{tikzcd}
		\]
		\item The diagrams 
		\[
		\begin{tikzcd}
			\mathbf{1}\times A \arrow[r,"\lambda_{A}"]\arrow[d] & A\arrow[d] \\
			\mathbf{1}\otimes A \arrow[r,"\ell_A"'] & A
		\end{tikzcd} \quad \text{and}\quad 
		\begin{tikzcd}
			A\times \mathbf{1} \arrow[r,"\rho_A"]\arrow[d] & A\arrow[d] \\
			A\otimes \mathbf{1} \arrow[r,"r_A"'] & A 
		\end{tikzcd}
		\]
		\item The diagrams
		\[
		\begin{tikzcd}
			A\times B\arrow[r,"\sigma_{A,B}"]\arrow[d] & B\times A\arrow[d]\\
			A\otimes B\arrow[r,"s_{A,B}"'] & B\otimes A
		\end{tikzcd}
		\]
	\end{itemize}
	We call $a_{A,B,C}$, $\ell_{A}$, $r_A$, and $s_{A,B}$ the associator, left unitor, right unitor, and braiding of $\otimes$, respectively.
\end{defn}

\begin{prop}
	The morphisms $a_{A,B,C}$, $\ell_{A}$, $r_A$, and $s_{A,B}$ are the components of natural isomorphisms. These natural isomorphisms define a symmetric monoidal structure $(\CSet_R,\otimes,\mathbf{1})$. 
\end{prop}

\begin{proof}
	Each of the morphisms is uniquely induced by the corresponding morphism in $\Set$, as are the morphisms giving the functoriality of the tensor products (e.g. $f\otimes(g\otimes h)$). Thus the naturality diagrams for the morphisms in $\Set$, together with the uniqueness guaranteed by Corollary \ref{cor:Set_maps_induce_tp_maps}, show that each of the transformations is natural. Identical arguments, applied to the inverse transformations of the structure maps in $\Set$ show that each transformation is a natural isomorphism. 
	
	Finally, given a diagram in $(\CSet_R,\otimes, \mathbf{1})$ comprised of associators, unitors, and braidings, we can lift it to a corresponding diagram of associators, unitors, and braidings in $(\Set,\times,\mathbf{1})$. Since, again, Corollary \ref{cor:Set_maps_induce_tp_maps} guarantees uniqueness of induced maps, the fact that these diagrams commute in $\Set$ means that the original diagram in $\CSet_R$ commutes. Thus $(\CSet_R,\otimes,\mathbf{1})$ is a symmetric monoidal category, as desired.
\end{proof}

\begin{rmk}
	Since the monoidal unit of $(\CSet_R,\otimes,\mathbf{1})$ is terminal, the monoidal structure of the previous proposition is semicartesian. 
\end{rmk}	

While much of the theory surrounding the convex tensor product likely feels familiar from the tensor product of vector spaces, there is one important property which contravenes this intuitive connection. 

\begin{lem}
	Let $X$, $Y$, and $Z$ be convex sets, and let $f:X\times Y\to Z$ be a convex map. Then $f$ is biconvex. 
\end{lem}

\begin{proof}
	We apply $f$ to convex combinations 
	\[
	\sum_{i=1}^n \alpha_i x_i \quad \text{and}\quad \sum_{j=1}^m \beta_j y_j 
	\]
	in $X$ and $Y$ respectively. Since, by definition, any constant convex combination of copies of $x$ is itself equal to $x$, we can rewrite these sums as 
	\[
	\sum_{i=1}^n\sum_{j=1}^m \alpha_i\beta_j x_i \quad \text{and}\quad \sum_{i=1}^n\sum_{j=1}^m \alpha_i\beta_j y_j. 
	\]
	Thus, we can write the pair 
	\[
	\left(\sum_{i=1}^n \alpha_i x_i,\sum_{j=1}^m \beta_j y_j \right)=\sum_{i=1}^n\sum_{j=1}^m \alpha_i\beta_j(x_i,y_j)
	\]
	and so, since $f$ is convex
	\[
	f\left(\sum_{i=1}^n \alpha_i x_i,\sum_{j=1}^m \beta_j y_j \right)=\sum_{i=1}^n\sum_{j=1}^m \alpha_i\beta_jf(x_i,y_j)
	\]
	that is, $f$ is biconvex.
\end{proof}

\begin{rmk}
	Note that not every biconvex map $f:X\times Y\to Z$ is convex. For example, the convex structure on the hom-sets $\CSet(X,Y)$ of the category $\CSet$ makes the composition in $\CSet$ into a biconvex map. To see that it is not convex, consider the sets $X=D_R(\{0,1\})$ and $Y=D_R(\{0,1,2,3\})$, the maps $g_0,g_1:X\to X$, where $g_0$ is the identity, and $g_1$ exchanges the delta distributions $\delta_0$ and $\delta_1$, and the maps $f_0,f_1$ from $X\to Y$, where $f_0(\delta_0)=\delta_0$, $f_0(\delta_1)=\delta_1$, $f_1(\delta_0)=\delta_2$, and $f_1(\delta_1)=\delta_3$. 
	
	We then investigate the image of $\delta_0$ under the map
	\[
	\left(\frac{1}{2}f_0+\frac{1}{2}f_1\right)\circ \left(\frac{1}{2}g_0+\frac{1}{2}g_1\right)
	\]
	Since the composition is biconvex, this is equal to 
	\[
	\frac{1}{4}\delta_0+ \frac{1}{4}\delta_1+\frac{1}{4}\delta_2+\frac{1}{4}\delta_3.  
	\]
	However, if the composition were convex, this would equivalently be 
	\[
	\left(\frac{1}{2}(f_0\circ g_0)+\frac{1}{2}(f_1\circ g_1)\right)(\delta_0)=\frac{1}{2}\delta_1+\frac{1}{2}\delta_3.
	\]
	Since these are unequal, we conclude that the composition map is \emph{not} convex. 
\end{rmk}

Specializing to the case $R=\RR_{\geq 0}$, we find that the convex tensor product allows us to reformulate the notion of convex category found in \cite{KharoofSimplicial}. 

\begin{defn}
	A \emph{$\CSet$-category} is a category enriched in the monoidal category $(\CSet,\otimes,\mathbf{1})$. We denote the category of $\CSet$-categories and $\CSet$-enriched functors by $\CSet\on{-}\Cat$. 
\end{defn}

\begin{defn}
	Let $(\sf{D}_R,\mathbf{\delta},\mathbf{\mu})$ be the monad on $\Cat$ of \cite[Corollary 3.7]{KharoofSimplicial}. Denote the category of algebras over this monad by $\Cat^{\sf{D}}$. 
\end{defn} 

Note that the set of objects of a $\sf{D}$-algebra $\scr{C}$ is an algebra over the identity monad on $\Set$, i.e., a Set. Thus, $\scr{C}$ is a category with no additional structure on the set of objects, which, for every $X,Y\in \on{Ob}(\scr{C})$, has a structure map $\pi^{\scr{C}(X,Y)}:D_R(\scr{C}(X,Y))\to \scr{C}(X,Y)$ which equips $\scr{C}(X,Y)$ with the structure of a convex set. Per  \cite[Proposition 3.12]{KharoofSimplicial}, such a structure on a 1-category defines a $\mathsf{D}$-structure on $\scr{C}$ if and only if the composition maps are biconvex. By the definition of a morphism of $\sf{D}$-algebras, a 1-functor $F:\scr{C}\to \scr{D}$ defines a morphism of $\sf{D}$-algebras if and only if it preserves the convex structure on the hom-sets.

\begin{thm}
	There is an isomorphism of categories between $\CSet\on{-}\Cat$ and $\Cat^{\sf{D}}$.
\end{thm}

\begin{proof} 	
	We define a functor $F:\CSet\on{-}\Cat\to \Cat^{\sf{D}_R}$. Given a $\CSet$-category $\scr{C}$, $F(\scr{C})$ is the category with the same objects and the same hom-sets, equipped with the structure map $\pi^{F(\scr{C})}$ which acts as the identity on objects. Since the composition in $\scr{C}$ is a convex map of the form
	\[
	\begin{tikzcd}
		\scr{C}(X,Y)\otimes \scr{C}(Y,Z)\arrow[r] & \scr{C}(X,Z) 
	\end{tikzcd}
	\]
	it uniquely determines a biconvex map
	\[
	\begin{tikzcd}
		\scr{C}(X,Y)\times \scr{C}(Y,Z)\arrow[r] & \scr{C}(X,Z) 
	\end{tikzcd}
	\]
	by Proposition \ref{prop:univ_prop_TP} which we define to be the composition in $F(\scr{C})$. The unit maps are the same as in $\scr{C}$. Since the associativity and unitality diagrams hold in $\scr{C}$, they also hold (by the uniqueness guaranteed by Proposition \ref{prop:univ_prop_TP}) for the biconvex maps defining composition in $F(\scr{C})$. 
	
	Given a functor $f:\scr{C}\to \scr{D}$, $F(f)$ acts identically to $f$ on both objects and hom-objects. As above, the fact that $f$ preserves the composition in $\scr{C}$ immediately implies that it preserves composition in $F(\scr{C})$ by the bijective correspondence between biconvex maps and maps from the tensor product. By construction, $F$ is fully faithful. 
	
	Moreover, given $(\sf{C},\pi^{\sf{C}})\in \Cat^{\sf{D}_R}$, we can define a $\CSet$-category $G(\sf{C})$ with the same objects and hom-objects, and composition with respect to the tensor product determined by the biconvex composition in $\sf{C}$. Since $F(G(\sf{C}))=\sf{C}$ and $G(F(\scr{C}))=\scr{C}$, this determines a bijection on objects. Thus $F$ is an isomorphism of categories, as desired. 
\end{proof}

\section{PROPs and operads for convexity}\label{sec:convprop}

In this section, we segue from the study of convex sets in terms of monads to the study of convex sets in terms of operads and PROPS. This approach is not original to us, and is prefigured by \cite{fritz2009presentation,fritz2015convex,Leinster_2021}. Of particular import, the PROP $\Conv$ which we define in this section was studied as a Lawvere theory by Fritz in \cite{fritz2009presentation,fritz2015convex} under the name $\sf{FinStoMap}^\op$. Similarly, the underlying operad of $\Conv$ which we denote by $\QConv$ was studied under the notation $\Delta$ in \cite[Ch. 12]{Leinster_2021}.

Our approach will treat the category $\Conv$ slightly more flexibly than these cited sources, in that we treat $\Conv$ as a PROP rather than a Lawvere theory, and study its algebras in a variety of symmetric monoidal categories. However, the key underlying idea is present in the literature --- $\Conv$-algebras in $\Set$ are the same as convex sets. 

\subsection{The $\on{Mat}_R$ and $\on{Conv}_R$ PROPs}

We begin by recalling some background on PROPs and defining the PROPs $\Mat_R$ and $\Conv_R$ which form the backbone of our treatment of abstract convexity.

\begin{defn} \label{def:PROP}
	A \emph{colored PROP}\footnote{This acronym stands for “products and permutations category“.} is a strict symmetric monoidal category $\P$ such that the monoid of objects is free. Let $\V$ be a symmetric monoidal category. A \emph{$\P$-algebra} in $\V$ is a monoidal functor $\P \to \V$ (\cite[Definition 1.2.10]{leinster2004higher})
\end{defn}

One way to interpret the notion of PROP is to say that it encodes abstract operations. We briefly discuss and recall some fairly standard terminology. Let $\P$ be a PROP with monoidal product $\otimes$ and unit $I$. The generators of the monoid of objects $\on{Ob}\P$ are called the \emph{colors}\footnote{This is traditional terminology from operad theory.} of $\P$. Given that $\on{Ob}\P$ is free, each object $\textbf{A} \in \P$ is determined by a sequence of colors $(A_1, \dots , A_m)$. The monoidal unit is represented by the empty sequence. 

We interpret a morphism $f : \textbf{A} \to \textbf{B}$ in $\P$ as an \emph{operation} with inputs the colours in the sequence $\textbf{A}$ and outputs the colors in the sequence $\textbf{B}$. Typically, such an operation is depicted as 

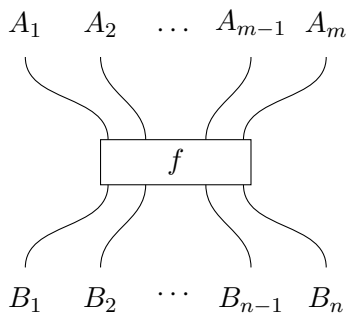
\begin{center}
	\begin{tikzpicture}
		\draw (-1,-0.3) rectangle (1,0.3);
		\path (0,0) node {$f$};
		
		\foreach \x/\lab/\lat in {-2/1/-0.9, -1/2/-0.4, 1/m-1/0.4,2/m/0.9}{
			\path (\x,1.4) node[label=above:$A_{\lab}$]{};
			\draw  (\x,1.4) to[out=-90,in=90] (\lat,0.3);
		}
		\path (0,1.4) node[label=above:$\cdots$] {};
		
		\foreach \x/\lab/\lat in {-2/1/-0.9, -1/2/-0.4, 1/n-1/0.4,2/n/0.9}{
			\path (\x,-1.4) node[label=below:$B_{\lab}$]{};
			\draw  (\x,-1.4) to[out=90,in=-90] (\lat,-0.3);
		}
		\path (0,-1.4) node[label=below:$\cdots$] {};
	\end{tikzpicture}
\end{center}

This way, the composition in $\P$ is depicted as

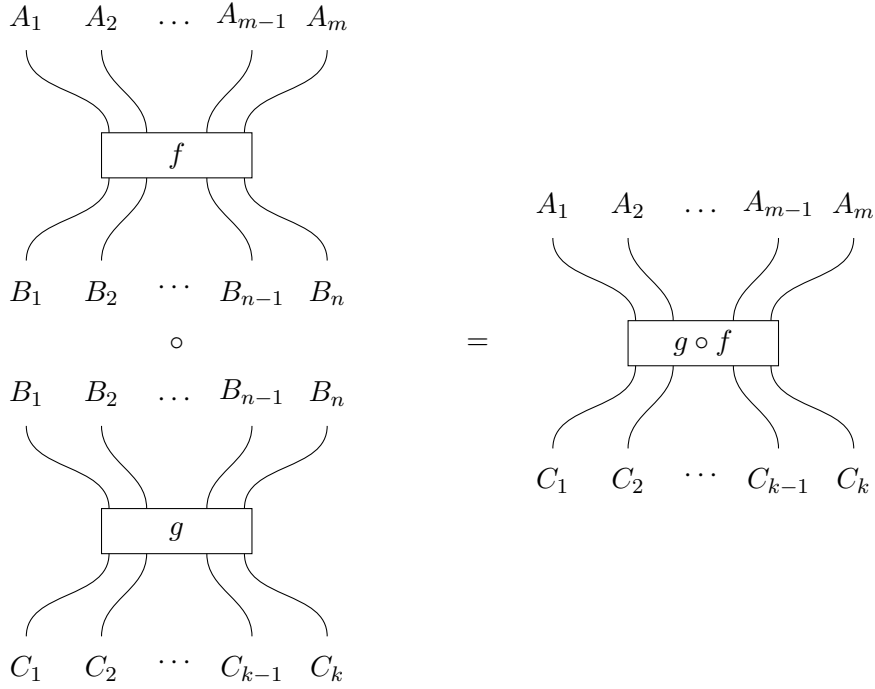
\begin{center}
	\begin{tikzpicture}
		\begin{scope}
			\draw (-1,-0.3) rectangle (1,0.3);
			\path (0,0) node {$f$};
			
			\foreach \x/\lab/\lat in {-2/1/-0.9, -1/2/-0.4, 1/m-1/0.4,2/m/0.9}{
				\path (\x,1.4) node[label=above:$A_{\lab}$]{};
				\draw  (\x,1.4) to[out=-90,in=90] (\lat,0.3);
			}
			\path (0,1.4) node[label=above:$\cdots$] {};
			
			\foreach \x/\lab/\lat in {-2/1/-0.9, -1/2/-0.4, 1/n-1/0.4,2/n/0.9}{
				\path (\x,-1.4) node[label=below:$B_{\lab}$]{};
				\draw  (\x,-1.4) to[out=90,in=-90] (\lat,-0.3);
			}
			\path (0,-1.4) node[label=below:$\cdots$] {};
		\end{scope}
		\path (0,-2.5) node {$\circ$};
		\begin{scope}[yshift=-5cm]
			\draw (-1,-0.3) rectangle (1,0.3);
			\path (0,0) node {$g$};
			
			\foreach \x/\lab/\lat in {-2/1/-0.9, -1/2/-0.4, 1/n-1/0.4,2/n/0.9}{
				\path (\x,1.4) node[label=above:$B_{\lab}$]{};
				\draw  (\x,1.4) to[out=-90,in=90] (\lat,0.3);
			}
			\path (0,1.4) node[label=above:$\cdots$] {};
			
			\foreach \x/\lab/\lat in {-2/1/-0.9, -1/2/-0.4, 1/k-1/0.4,2/k/0.9}{
				\path (\x,-1.4) node[label=below:$C_{\lab}$]{};
				\draw  (\x,-1.4) to[out=90,in=-90] (\lat,-0.3);
			}
			\path (0,-1.4) node[label=below:$\cdots$] {};
		\end{scope}
		\path (4,-2.5) node {$=$};
		\begin{scope}[yshift=-2.5cm,xshift=7cm]
			\draw (-1,-0.3) rectangle (1,0.3);
			\path (0,0) node {$g\circ f$};
			
			\foreach \x/\lab/\lat in {-2/1/-0.9, -1/2/-0.4, 1/m-1/0.4,2/m/0.9}{
				\path (\x,1.4) node[label=above:$A_{\lab}$]{};
				\draw  (\x,1.4) to[out=-90,in=90] (\lat,0.3);
			}
			\path (0,1.4) node[label=above:$\cdots$] {};
			
			\foreach \x/\lab/\lat in {-2/1/-0.9, -1/2/-0.4, 1/k-1/0.4,2/k/0.9}{
				\path (\x,-1.4) node[label=below:$C_{\lab}$]{};
				\draw  (\x,-1.4) to[out=90,in=-90] (\lat,-0.3);
			}
			\path (0,-1.4) node[label=below:$\cdots$] {};
		\end{scope}
	\end{tikzpicture}
\end{center}

\noindent{}and referred to as \emph{vertical composition} of operations. Similarly, given two operations $f: \textbf{A} \to \textbf{B}$ and $g : \textbf{B} \to \textbf{C}$, the formation of the monoidal product $f \otimes g$ is referred to as the \emph{horizontal composition} of $f$ and $g$ when ``placed next to each other''

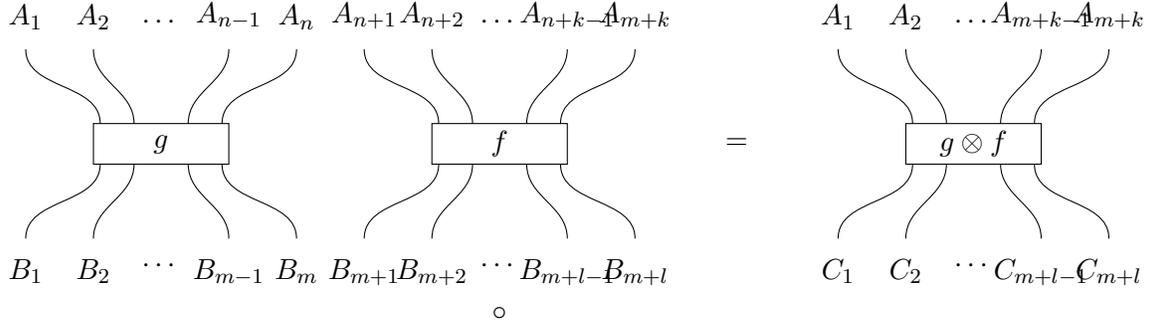
\begin{center}
	\begin{tikzpicture}[scale=0.9]
		\begin{scope}
			\draw (-1,-0.3) rectangle (1,0.3);
			\path (0,0) node {$f$};
			
			\foreach \x/\lab/\lat in {-2/n+1/-0.9, -1/n+2/-0.4, 1/n+k-1/0.4,2/m+k/0.9}{
				\path (\x,1.4) node[label=above:$A_{\lab}$]{};
				\draw  (\x,1.4) to[out=-90,in=90] (\lat,0.3);
			}
			\path (0,1.4) node[label=above:$\cdots$] {};
			
			\foreach \x/\lab/\lat in {-2/m+1/-0.9, -1/m+2/-0.4, 1/m+l-1/0.4,2/m+l/0.9}{
				\path (\x,-1.4) node[label=below:$B_{\lab}$]{};
				\draw  (\x,-1.4) to[out=90,in=-90] (\lat,-0.3);
			}
			\path (0,-1.4) node[label=below:$\cdots$] {};
		\end{scope}
		\path (0,-2.5) node {$\circ$};
		\begin{scope}[xshift=-5cm]
			\draw (-1,-0.3) rectangle (1,0.3);
			\path (0,0) node {$g$};
			
			\foreach \x/\lab/\lat in {-2/1/-0.9, -1/2/-0.4, 1/n-1/0.4,2/n/0.9}{
				\path (\x,1.4) node[label=above:$A_{\lab}$]{};
				\draw  (\x,1.4) to[out=-90,in=90] (\lat,0.3);
			}
			\path (0,1.4) node[label=above:$\cdots$] {};
			
			\foreach \x/\lab/\lat in {-2/1/-0.9, -1/2/-0.4, 1/m-1/0.4,2/m/0.9}{
				\path (\x,-1.4) node[label=below:$B_{\lab}$]{};
				\draw  (\x,-1.4) to[out=90,in=-90] (\lat,-0.3);
			}
			\path (0,-1.4) node[label=below:$\cdots$] {};
		\end{scope}
		\path (3.5,0) node {$=$};
		\begin{scope}[xshift=7cm]
			\draw (-1,-0.3) rectangle (1,0.3);
			\path (0,0) node {$g\otimes f$};
			
			\foreach \x/\lab/\lat in {-2/1/-0.9, -1/2/-0.4, 1/m+k-1/0.4,2/m+k/0.9}{
				\path (\x,1.4) node[label=above:$A_{\lab}$]{};
				\draw  (\x,1.4) to[out=-90,in=90] (\lat,0.3);
			}
			\path (0,1.4) node[label=above:$\cdots$] {};
			
			\foreach \x/\lab/\lat in {-2/1/-0.9, -1/2/-0.4, 1/m+l-1/0.4,2/m+l/0.9}{
				\path (\x,-1.4) node[label=below:$C_{\lab}$]{};
				\draw  (\x,-1.4) to[out=90,in=-90] (\lat,-0.3);
			}
			\path (0,-1.4) node[label=below:$\cdots$] {};
		\end{scope}
	\end{tikzpicture}
\end{center}

Lastly, we may think of the symmetries in the monoidal structure of $\P$ as acting on operations by “permuting inputs and outputs". More precisely, let $\Sigma_m$ be the symmetric group in $m$ variables. Given an object $\textbf{A} = (A_1, \dots, A_m)$ of length $m$ and a symmetry $\sigma \in \Sigma_m$, we have a structure isomorphism in $\P$ of the form $S_\sigma : \textbf{A}\sigma \to \textbf{A}$, where $\textbf{A}\sigma = (A_{\sigma(1)}, \dots , A_{\sigma(m)})$. 

This way, given $f : \textbf{A} \to \textbf{B}$, $\sigma \in \Sigma_{|\textbf{A}|}$ and $\tau \in \Sigma_{|\textbf{B}|}$, we may form the operation $\tau f \sigma = S_\tau \circ f \circ S_\sigma : \textbf{A}\sigma \to \tau \textbf{B}$. More precisely, we may say that these symmetries equip the set $\P(\textbf{A}, \textbf{B})$ with the structure of a $\Sigma_\textbf{B}-\Sigma_\textbf{A}$-biset. Diagrammatically, we depict the formation on $\tau f \sigma$ as 

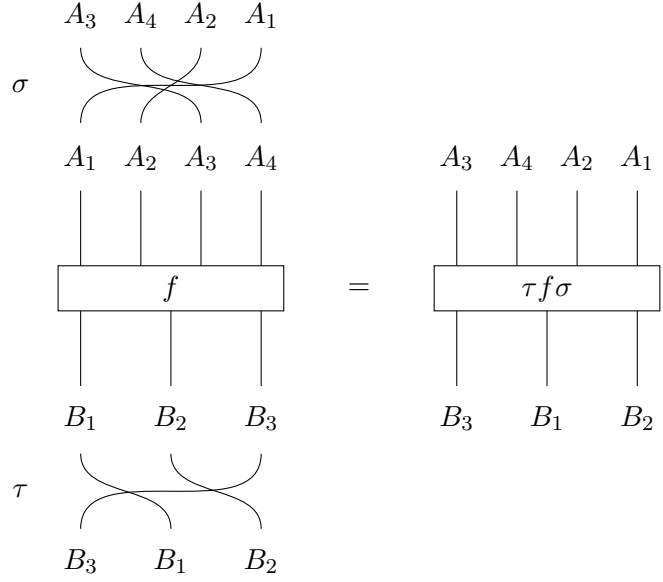
\begin{center}
	\begin{tikzpicture}
		\draw (-1.5,-0.3) rectangle (1.5,0.3);
		\path (0,0) node {$f$};
		\foreach \x/\lab in {-1.2/1,-0.4/2,0.4/3,1.2/4}{
			\path (\x,1.3) node[label=above:$A_{\lab}$]{};
			\draw (\x,1.3) to (\x,0.3);
		}
		\foreach \x/\lab in {-1.2/1,0/2,1.2/3}{
			\path (\x,-1.3) node[label=below:$B_{\lab}$]{};
			\draw (\x,-1.3) to (\x,-0.3);
		}
		\draw (-1.2,3.2) to[out=-90,in=90] (0.4,2.2);
		\draw (-0.4,3.2) to[out=-90,in=90] (1.2,2.2);
		\draw (0.4,3.2) to[out=-90,in=90] (-0.4,2.2);
		\draw (1.2,3.2) to[out=-90,in=90] (-1.2,2.2);
		\path (-2, 2.7) node {$\sigma$};
		\foreach \x/\lab in {-1.2/3,-0.4/4,0.4/2,1.2/1}{
			\path (\x,3.2) node[label=above:$A_{\lab}$]{};
		}
		\draw (-1.2,-2.2) to[out=-90,in=90] (0,-3.2); 
		\draw (0,-2.2) to[out=-90,in=90] (1.2,-3.2); 
		\draw (1.2,-2.2) to[out=-90,in=90] (-1.2,-3.2);
		\foreach \x/\lab in {-1.2/3,0/1,1.2/2}{
			\path (\x,-3.2) node[label=below:$B_{\lab}$]{};
		} 
		\path (2.5,0) node {$=$};
		\path (-2, -2.7) node {$\tau$};
		\begin{scope}[xshift=5cm]
			\draw (-1.5,-0.3) rectangle (1.5,0.3);
			\path (0,0) node {$\tau f\sigma$};
			\foreach \x/\lab in {-1.2/3,-0.4/4,0.4/2,1.2/1}{
				\path (\x,1.3) node[label=above:$A_{\lab}$]{};
				\draw (\x,1.3) to (\x,0.3);
			}
			\foreach \x/\lab in {-1.2/3,0/1,1.2/2}{
				\path (\x,-1.3) node[label=below:$B_{\lab}$]{};
				\draw (\x,-1.3) to (\x,-0.3);
			}
		\end{scope}
	\end{tikzpicture}
\end{center}

It is possible to rewrite the definition of PROP in terms of the above elements by saying that a PROP consists of colours, operations, vertical and horizontal composition and bisymmetries which satisfy some compatibility conditions (\cite{hackney2015category}). 

In this vein, an algebra $\P \to \V$ is simply a concrete realization in $\V$ of the abstract operations recorded in $\P$. Such algebras form a category which we denote $\P (\V)$. The morphisms in this category are monoidal natural transformations ([\textbf{Leinster}]). 

A PROP $\P$ is called \emph{monocolored} in case the set of colours is the singleton set, or equivalently if the monoid of objects of $\P$ as a monoidal category  is the monoid of natural numbers $\bb{N}$. In such case, we say that an object $A \in \V$ has the structure of a $\P$-algebra in case there is a monoidal functor $\P \to \V$ which maps the generator object $1$ to $A$. In this vein, the object $A$ is equipped with extra structure: for each $p \in \P(m,n)$ a map  $\mu_p : A^{\otimes m } \to A^{\otimes n}$ which compose as prescribed by $\P$. This way, a morphism of $\P$-algebras $f : A \to B$ is a morphism in $\V$ which preserves the algebra structure.

\begin{defn} \label{denf:Mat}
	
	There is a monocolored PROP $\Mat_R$ with:
	\begin{itemize}
		\item Set of operations $\Mat(m,n)$ being the set of $n\times m$ matrices with entries in the semiring $R$. We allow for the case $n = 0$. When $n=0$, a $0\times m$ matrix  is interpreted as a unique ``forgetting'' operation, or the linear transformation $R^m\to 0$ to the terminal vector space. 
		\item Composition provided by matrix multiplication.
		
		\item Horizontal composition provided by the direct sum of matrices
		\[
		P \oplus Q = \begin{bmatrix}
			P & 0\\
			0 & Q
		\end{bmatrix}.
		\]
		\item For an $m \times n$ matrix $P$, $\sigma \in \Sigma_m$ and $\tau \in \Sigma_n$, the matrix $\tau P \sigma$ being formed by permuting rows and columns.
	\end{itemize}
	
	In case $R$ is a field, then the vector space $R$ has the structure of a $\Mat$-algebra, provided by the linear transformations $R^m \to R^n$ corresponding to the matrices in $\Mat(m,n)$ (with respect to a chosen basis).
\end{defn}

\begin{example}\label{eg:VSs_As_algebras}
	We consider the most general version of the matrix prop $\sf{Mat}_R$ in which we allow the unique nullary operations in $\sf{Mat}(0,n)$ as well as the unique conullary operations in $\sf{Mat}(n,0)$. In the case where $R$ is a commutative ring, we unpack the category $\sf{Mat}_R(\Set)$ of $\Set$-valued algebras over $\sf{Mat}_R$. 
	
	Given a $R$-module $V$, we can construct an algebra $A_V$ in $\Set$ by defining 
	\[
	A_V(n):=V^{\times n}
	\]
	and, for an operation $M\in \Mat(m,n)$, setting 
	\[
	A_V(M)(v_1,\ldots,v_m):= \left(\sum_{j=1}^m M_{1,j}v_j,\ldots,\sum_{j=1}^m M_{n,j}v_j\right).
	\]
	Where the convention holds that the empty sum and the empty tuple both represent $0$. It is clear that this definition respects matrix multiplication and the direct sum, yielding an algebra.
	
	Any algebra $A:\Mat_R\to \Set$ will, in particular, yield a set $V:=A(1)$ and the following data: 
	\begin{itemize}
		\item For every $n\geq0$, the diagonal map $\delta_n:V\to V^{\times n}$. 
		\item A map $0: \ast \to V$ corresponding to the unique nullary operation. 
		\item For every $a\in k$, a map $\lambda_a: V\to V$ corresponding to the operation $(a)\in \Mat_k(1,1)$. (Note here that unitality means that $\lambda_1$ must be the identity.) 
		\item A binary operation 
		\[
		\begin{tikzcd}
			+ :&[-3em] V\times V\arrow[r] & V
		\end{tikzcd}
		\]
		corresponding to the matrix $(1 \; 1)\in \Mat_R(2,1)$. 
	\end{itemize}
	A little elbow grease shows that every operation in $\Mat_R$ can be built out of these, so that the data listed above actually determine the algebra $A$. Note that since $\Mat_R$ contains the commutative operad, the operation $+$ and the nullary operation $0$ in fact endow $V$ with the structure of a commutative monoid. Moreover, the operations $\lambda_a$ determine an action of $R$ on $V$. We aim to show that this, in fact, endows $V$ with the structure of a $R$-vector space. 
	
	The two distributivity axioms, the multiplicativity of the $R$-action, the identity-preserving nature of the $R$-action, and the fact that $(V,+,0)$ is a commutative monoid all follow immediately from the structure of $\Mat_k$. The one difficulty is showing that the map $\lambda_{-1}$ does, in fact, send an element $v\in V$ to an additive inverse under $+$. Since 
	\[
	\begin{pmatrix}
		1 & -1 
	\end{pmatrix} \begin{pmatrix}
		1\\
		1 
	\end{pmatrix}=\begin{pmatrix}
		0
	\end{pmatrix}
	\]
	we see that, for any $v\in V$
	\[
	v+\lambda_{-1}(v)=\lambda_0(v). 
	\]
	Thus, to see that $V$ is indeed endowed with a vector space structure, it will suffice to see that $\lambda_0$ factors through the unit map $0:\ast\to V$. However, we have that 
	\[
	\begin{pmatrix}
		1 & 1 & 0\\
		0 & 0 & 1 
	\end{pmatrix}
	\begin{pmatrix}
		1 & 0 & 0 \\
		0 & 0 & 0 \\
		0 & 0 & 1 
	\end{pmatrix}\begin{pmatrix}
		1 & 0  \\
		0 & 1 \\
		0 & 1
	\end{pmatrix}= \begin{pmatrix}
		1 & 0 \\
		0& 1 
	\end{pmatrix}
	\] 
	Or, rephrased in terms of the structure maps we have extracted, $\forall v,w\in V$, 
	\[
	(\lambda_1(v)+\lambda_0(w),\lambda_1(w))=(v,w). 
	\]
	Since $\lambda_1$ is the identity, this implies that for any $w\in V$, 
	\[
	v+\lambda_0(w)=v 
	\]
	for all $v\in V$. Thus, the uniqueness of units in a monoid shows that $\lambda_0(w)$ is constant on the monoidal unit $0$, as desired. We thus see that every algebra over $\Mat_R$ has an underlying $R$-module, which we will now denote by $V(A)$. It is easy to check that for a vector space $W$, $V(A_W)=W$, and for an algebra $B$, $A_{V(B)}=B$. It is similarly easy, though tedious, to check that these two constructions are functorial and the above isomorphisms are natural, and thus we have an equivalence of categories $\sf{Mod}_R\cong \Mat_R(\Set)$. In particular, in the case of a field $k$, we obtain an equivalence $\sf{Vect}_k\cong \Mat_k(\Set)$.
\end{example}

\begin{example}
	
	Following on from Example \ref{eg:VSs_As_algebras}, we can make the following observation. We can view the 1-category $\Cat$ of small categories as the full subcategory $\sf{Seg}\subset\Set_\Delta$ on the Segal sets. Since this category is closed under limits, the category $\Mat_k(\Cat)$ is equivalent to the category of simplicial objects in $\Mat_k(\Set)$ whose underlying simplicial sets are Segal. However, since limits in $\Vect_k$ are computed as limits of the underlying diagrams of sets, this means that the category $\Mat_k(\Cat)$ is equivalent to the category of $(\Vect_k,\times)$-internal categories. Such categories have already appeared in the literature, in particular under the name \emph{$2$-vector spaces} in \cite{BaezCranz} where they are used to characterize 2-term $L_\infty$ algebras as $2$-vector spaces equipped with a (categorified) Lie bracket. 
	
\end{example}

\begin{example}[Topological vector spaces]
	The computations in example \ref{eg:VSs_As_algebras} can be easily repurposed to show that, for a field $k$ a  $\Mat_k$-algebra in $\sf{Top}$ is the same thing as a vector space $V$ equipped with a topology such that the monoid structure and individual scalar multiplication operations are continuous. 
	
	Note that this is not quite the same as a topological vector space, in that we do not require that the map $k\times V\to V$ is continuous. However, in the case where $k$ is a topological field, the monoidal category $\Mat_k$ is canonically $\sf{Top}$-enriched, via the identification of $\Mat_k(M,n)$ with $k^{m\times n}$. Letting $\sf{ExpTop}\subset \sf{Top}$ be the full subcategory of exponentiable spaces, we can view $\sf{ExpTop}$ as a $\sf{Top}$-enriched category with a monoidal structure given by the cartesian product. Then enriched monoidal functors $\Mat_k\to \sf{ExpTop}$ are topological vector spaces whose underlying topological space is exponentiable. Note that this means that such enriched monoidal functors cannot capture infinite-dimensional Hausdorff topological vector spaces, since the latter are not locally compact, and thus not exponentiable. 
\end{example}

\begin{rmk}[Related structures]
	
	For any PROP, the system of operations which have as target a single colour forms a colored \textit{operad}. Dually, operations whose source is a single colour form a \textit{cooperad}. By forgetting the horizontal composition of operations we obtain a colored \textit{properad} [\cite[Chapter 3]{hackney2015infinity}]. In case $\P$ is monocolored, we denote the underlying operad $\P(-, 1)$ and the underlying cooperad $\P(1 , -)$. 
	
\end{rmk}

While the PROP $\Mat_R$ is itself of substantial interest, our aim in this work is to explore convexity, and so we specialize to a sub-PROP of $\Mat_R$ more suited to this purpose. 

\begin{defn}[Convex matrix] \label{defn: ConvMat}
	
	We say that a matrix is \emph{convex} if the sum of entries in each of its rows is 1. That is, for $M\in \Mat_R(m,n)$, we say that $M$ is \emph{convex} if, for every $1\leq i\leq n$, the sum 
	\[
	\sum_{j=1}^m M_{i,j}=1. 
	\]
	Note that, for the unique $0\times m$ matrix, this condition is vacuous, so that we still regard the $0\times m$ matrix as convex.    
\end{defn}

\begin{example}[The convexity PROP] \label{expl : Conv}
	
	There is a PROP $\Conv_R$ whose operations are convex matrices. Given that convex matrices are closed under direct sum, $\Conv_R$ can be defined as the subcategory of $\Mat_R$ consisting of convex matrices. 
	
\end{example}

Note that the transposes of convex matrices appear in \cite{fritz2009presentation} under the name \emph{stochastic matrices}, and the PROP $\Conv$ under the name $\sf{FinStoMap}^\op$. 

Let us unpack a little of the structure of $\Conv$. First, observe that for each $n \geq 1$ there is a unique matrix $C_n \in \Conv(1,n)$, the column of height $n$ with all entries being $1$
\[
C_n =  \begin{bmatrix}
	1 \\
	\vdots \\
	1
\end{bmatrix}
\] 
Hence, we have an isomorphism
$$\Conv(1 , -) \cong \Comm^{\op}$$
In other words, the underlying cooperad in $\Conv$ is the commutative cooperad. The latter is the terminal cooperad, i.e. it is monocolored and it has a unique $n$-ary cooperation for each $n$. The algebras of $\Comm$ in a monoidal category $\V$ are commutative monoids in $\V$, while, by duality, algebras in $\V$ of the cooperad $\Comm^\op$ are commutative comonoids in $\V$.

\begin{defn}
	The \emph{operad for quasiconvexity} $\QConv$ is the symmetric endomorphism operad of $1$ in the symmetric monoidal category $\on{Conv}$. More precisely, the is the monochromatic symmetric operad with 
	\[
	\QConv(m)=\left\lbrace\vec{\alpha}\in (\RR_{\geq 0})^m\;\middle|\; \sum_i\alpha_i =1\right\rbrace.
	\]
	A \emph{quasiconvex object} in a symmetric monoidal category $\scr{C}$ is a $\QConv$ algebra in $\scr{C}$. 
\end{defn}

\begin{rmk}
	The operad $\QConv$ is precisely the operad $\Delta$ considered by Leinster in \cite[Ch. 12]{Leinster_2021}.
\end{rmk}

Notice that we have an isomorphism 
$$\Conv(-, 1) \cong \QConv$$
i.e. $\QConv$ is the underlying operad of the PROP $\Conv$. The algebras of $\QConv$ in the category $\Set$ may be regarded as \emph{quasiconvex sets}. The latter term indicates, for some algebra $A \in \QConv(\Set)$, the presence of a convex structure, even though the full-structure of a convex set is not present. More precisely, for each $\vec{\alpha} = (\alpha_1, \dots , \alpha_m) \in \QConv(m)$ we have a structure map
\[
\begin{tikzcd}
    \mu_{\vec{\alpha}} : &[-3em] A^m \arrow[r] & A \\ [-2em]
& (a_1, \dots , a_m) \arrow[r, mapsto] & \sum_i \alpha_i a_i
\end{tikzcd}
\]
which provides convex combinations for the set $A$. However, for an element $a \in A$, we do not have $\sum_i \alpha_i a = a$ by virtue of $A$ being a $\QConv$-algebra.

\subsection{Convex objects}

Having defined the PROP $\Conv_R$, we now connect the theory of $\Conv_R$-algebras to the existing theory of convex sets.

\begin{defn} \label{def : convobj}
	
	Let $\C$ be a category with finite products. An object $A \in \C$ is said to be \emph{$R$-convex} if it has a $\Conv_R$-algebra structure in the cartesian monoidal structure. We denote by $\Conv_R(\C)$ the category of $R$-convex objects.
	
	More tersely, an \emph{$R$-convex object} is a $\Conv_R$-algebra in a Cartesian monoidal structure.
	
\end{defn}

\begin{thm} \label{thm : convset}
	 There is an isomorphism of categories
	$$\CSet \cong \Conv(\Set)$$
\end{thm}

\begin{proof}
	The statement of this theorem is equivalent to that of Proposition 3.7 in [\textbf{Fritz 2015}]
	
\end{proof}

\begin{example}
    Let $\sf{Top}$ be the category of topological spaces. Then, $\Conv(\sf{Top})$ is the category of \emph{convex topological spaces}. In particular, for a finite set $X$, the set of distributions $D(X)$ with subspace topology, by virtue of inclusion into $\RR^{|X|}$, is an example since taking convex manipulations are continuous.
\end{example}

\begin{example}
	
	Let $\Set_{\on{rel}}$ be the category of sets which are equipped with a relation. More precisely, the objects in this category are pairs $(A, r)$, where $A$ is a set and $r \subseteq A \times A$, and morphisms are functions between sets which preserve the structure relation. Then, convex objects in $\Set_{\on{rel}}$ are precisely convex sets equipped with a convex relation in the sense of Definition \ref{def : convrel}. 
	
\end{example}

	\begin{example}
	As in the example of 2-vector spaces, we see immediately that $\Conv_{R}$-algebras in $\Cat$ are equivalent to $\CSet$-internal categories. 
\end{example}

\begin{example}
	Given an inclusion $R\to S$ of semirings, there is an obvious commutative diagram 
	\[
	\begin{tikzcd}
		\Conv_R \arrow[r]\arrow[d] & \Conv_{S}\arrow[d] \\
		\Mat_R \arrow[r] & \Mat_S 
	\end{tikzcd}
	\]
	of PROPS. For any  complete category $\scr{C}$, These induce forgetful functors on categories of algebras
	\[
	\begin{tikzcd}
		\Mat_S(\scr{C})\arrow[d] \arrow[r] & \Mat_R(\scr{C})\arrow[d]\\
		\Conv_S(\scr{C})\arrow[r] & \Conv_R(\scr{C})
	\end{tikzcd}
	\] 
	which preserve limits. Applying the adjoint functor theorem, all of these functor have left adjoints.
\end{example}

	

\section{The convex Grothendieck construction}\label{sec:GC}

We now turn to considering Grothendieck constructions for functors valued in convex sets. In the convex setting, there is an obvious, na\"ive approach to Grothendieck constructions, following from a direct application of the PROP $\Conv_R$. We will briefly spell out this Grothendieck construction and its proof, however, the examples which motivate our investigations of convexity are such that this is insufficient for our purposes. 

Instead, the key feature of our convex Grothendieck construction will be it's interaction with generalizations of monoidal structures. Along the lines of \cite{MoellerVasilakopoulou}, we will show a correspondence between a certain kind of lax monoidal functors and a certain kind of monoidal fibrations. 

\subsection{The na\"ive convex Grothendieck construction}

Immediately following from the construction of the convexity PROP, we can provide a Grothendieck construction for $\CSet$-valued functors. 

\begin{defn}
	Let $\pi:\scr{D}\to \scr{C}$ be a functor between small categories. A morphism $f:x\to y$ in $\scr{D}$ is said to be \emph{coCartesian}\footnote{Sometimes called opcartesian in the literature.} if, for every diagram in $\scr{D}$ and $\scr{C}$ 
	\[
	\begin{tikzcd}
		& y & \\
		x\arrow[ur,"f"]\arrow[rr,"h"'] & & z \\
		&\rotatebox{-90}{$\mapsto$} & \\
		& \pi(y)\arrow[dr,"h"] & \\
		\pi(x)\arrow[ur,"\pi(f)"]\arrow[rr,"(\pi(h))"'] & & \pi(z) \\
	\end{tikzcd}
	\]
	such that the bottom triangle commutes in $\scr{C}$, there exists a unique $\widetilde{h}:y\to z$ in $\scr{D}$ making the top triangle commute, and such that $\pi(\widetilde{h})=h$.  
\end{defn}

\begin{defn}
	A functor $\pi:\scr{D}\to \scr{C}$ is said to be a \emph{discrete fibration} if the strict fibres $\pi^{-1}(c)$ are discrete (i.e., sets) for all $c\in \on{Ob}(\scr{C})$ and every morphism of $\scr{D}$ is coCartesian. A morphism of discrete fibrations from $\pi:\scr{D}\to\scr{C}$ to $\rho:\scr{E}\to\scr{C}$ is a functor $F:\scr{D}\to\scr{E}$ making the diagram 
	\[
	\begin{tikzcd}
		\scr{D}\arrow[dr,"\pi"'] \arrow[rr,"F"] & & \scr{E} \arrow[dl,"\rho"]\\
		& \scr{C} & 
	\end{tikzcd}
	\]
	commute. We denote the category of discrete fibrations by $\sf{DFib}(\scr{C})$. 
\end{defn}

Given a functor $F:\scr{C}\to \Set$, we can construct a discrete fibration $\int_{\scr{C}}F\to \scr{C}$ as follows. Define $\int_{\scr{C}}F$ to have objects given by pairs $(c,x)$, where $c$ is an object of $\scr{C}$ and $x\in F(c)$. A morphism $(c,x)\to (d,y)$ is a morphism $f:c\to d$ in $\scr{C}$ such that $F(f)(x)=y$.  It is not hard to check that this construction defines a functor 
\[
\begin{tikzcd}
	\int_{\scr{C}}:&[-3em] \Fun(\scr{C},\Set)\arrow[r] & \sf{DFib}(\scr{C}).
\end{tikzcd}
\]
In the case we are concerned with, this is the famous \emph{Grothendieck construction}, the utility of which lies in the following result.

\begin{thm}[The Grothendieck construction]
	For any small category $\scr{C}$, the functor 
	\[
	\begin{tikzcd}
		\int_{\scr{C}}:&[-3em] \Fun(\scr{C},\Set)\arrow[r] & \sf{DFib}(\scr{C})
	\end{tikzcd}
	\]
	is an equivalence of categories. 
\end{thm}

To leverage the Grothendieck construction in our context, we need a lemma about monoidal categories. 

\begin{lem}\label{lem:sm_into_functor_cat}
	Let $(\scr{C},\otimes,E)$ be a small symmetric monoidal category, $\scr{I}$ a small category, and consider $\Set$ and $\Fun(\scr{I},\Set)$ to be equipped with the Cartesian symmetric monoidal structures. There is an isomorphism 
	\[
	\Fun^{\on{sm}}(\scr{C},\Fun(\scr{I},\Set))\cong \Fun(\scr{I},\Fun^{\on{sm}}(\scr{C},\Set)).
	\]
\end{lem} 

\begin{proof}
	If we remove the superscripts ``$\on{sm}$'', the isomorphism 
	\[
	\Fun^{}(\scr{C},\Fun(\scr{I},\Set))\cong \Fun(\scr{I},\Fun^{}(\scr{C},\Set))
	\]
	is simply a two-fold application of the Cartesian closedness of $\Cat$. We will denote the image of a functor $F:\scr{C}\to \Fun(\scr{I},\Set)$ under this isomorphism by $\widetilde{F}$. 
	
	The structure of $F:\scr{C}\to \Fun(\scr{I},\Set)$ being a \emph{symmetric monoidal} functor amounts having an isomorphism $\ast\cong F(E)$ and a natural isomorphism $F(x)\times F(y)\cong F(x\otimes y)$ satisfying coherence conditions with the associators, unitors, and braiding of $\Fun(\scr{I},\Set)$.
	
	Again by the Cartesian closure of $\Cat$, a natural transformation between functors $\scr{C}\times \scr{C}\to \Fun(\scr{I},\Set)$ is the same thing as a functor 
	\[
	\scr{I}\to \Fun(\scr{C}\times \scr{C}\times [1],\Set).
	\]
	Thus, the data of the symmetric monoidal structure provide equivalent data for symmetric monoidal structures on each functor $\widetilde{F}(i)$, together with the requirement that, for $f:i\to j$ the induced maps $\widetilde{F}(f)$ commute with these data. That is, so long as these data satisfy the coherence conditions of symmetric monoidal functors, the functor $\widetilde{F}$ will take values in monoidal natrual transformations. However, the coherence conditions for $F$ are checked objectwise in the functor category, meaning that the functors $\widetilde{F}(i)$ are, indeed, symmetric monoidal. This completes the proof. 
\end{proof}

\begin{cor}\label{cor:naiveGC}
	The Grothendieck construction induces an equivalence of categories
	\[
	\begin{tikzcd}
		\cint_\scr{C}:&[-3em]\Fun(\scr{C},\CSet) \arrow[r] & \Conv(\sf{DFib}(\scr{C})).
	\end{tikzcd}
	\]
\end{cor}
\begin{proof}
		Since the Grothendieck construction is an equivalence of categories, it induces an equivalence of $\Conv$-algebras
		\[
		\Conv(\Fun(\scr{C},\Set))\simeq \Conv(\sf{DFib}(\scr{C})). 
		\]
		By Lemma \ref{lem:sm_into_functor_cat}, the former is equivalent to 
		\[
		\Fun(\scr{C},\Conv(\Set))\cong \Fun(\scr{C},\CSet).\qedhere
		\] 
\end{proof}	

\begin{rmk}
	Convex objects in discrete fibrations are simply discrete fibrations $\pi:\scr{D}\to \scr{C}$ equipped with convex structures on their fibres together with convex structures on the sets  $\pi^{-1}(f)$ for every $f:c\to d$ in $\scr{C}$ such that the source and target maps satisfy 
	\[
	s\left(\sum_{i}\alpha_i \phi_i\right)=\sum_i\alpha_i s(\phi_i), \qquad  t\left(\sum_{i}\alpha_i \phi_i\right)=\sum_i\alpha_i t(\phi_i) 
	\]
	for $\{\phi_i\}\subset \pi^{-1}(f)$
	and the identity maps satisfy 
	\[
	\on{Id}_{\sum_i \alpha_i x_i}=\sum_{i}\alpha_i \on{Id}_{x_i} 
	\]
	for $\{x_i\}\subset \pi^{-1}(c)$. In this sense, $\cint_{\scr{C}}$ simply remembers the convex structures on the sets $F(c)$, and remembers $f:(c,\sum {\alpha_i}(x_i))\to (d,\sum \alpha_i y_i)$ as the corresponding combination of the morphisms $f:(c,x_i)\to (d,y_i)$. 
\end{rmk}

As we will have cause to refer to it later, we turn this observation into a definition.

\begin{defn}
	Let $\pi:\scr{D}\to \scr{C}$ be a discrete fibration. A \emph{fibrewise convex structure} on $\pi$ consists of 
	\begin{itemize}
		\item A convex structure on $\pi^{-1}(c)$ for each $c\in \scr{C}$. 
		\item A convex structure on $\pi^{-1}(f)$ for each $f:c\to d$ in $\scr{C}$. 
	\end{itemize}
	Such that the following three conditions are satisfied for any convex vector $\vec{\alpha}$, any $f:c\to d$ in $\scr{C}$, and any collection $\{\phi_i\}\subset \pi^{-1}(f)$. 
	\begin{itemize}
		\item The source map $s$ of $\scr{D}$ satisfies 
		\[
		s\left(\sum_{i}\alpha_i \phi_i\right)=\sum_i\alpha_i s(\phi_i)
		\]
		\item The target map $t$ of $\scr{D}$ satisfies 
		\[
		t\left(\sum_{i}\alpha_i \phi_i\right)=\sum_i\alpha_i t(\phi_i) 
		\]
		\item The identity maps satisfy 
		\[
		\on{Id}_{\sum_i \alpha_i x_i}=\sum_{i} \alpha_i \on{Id}_{x_i} 
		\] 
	\end{itemize}
	Morphisms of fibrewise convex fibrations are morphisms of discrete fibrations which induce convex maps on fibres over objects and morphisms. We denote the category of fibrewise convex fibrations as $\sf{FCFib}_{\scr{C}}$
\end{defn}

\subsection{$\on{QConv}_R$-monoidal categories}

While the na\"ive convex Grothendieck construction is of interest in itself, we will shortly see that many natural examples of Grothendieck constructions of functors valued in convex sets are fundamentally connected to the convex tensor product and the generalizations of monoidal structures described in \cite[\S 2]{OMonGroth}. One key set of examples will be generalized monoidal structures governed by the operad $\on{QConv}_R$ itself. The theory of such generalized monoidal structures is laid out in detail in the paper \cite{OMonGroth} of the last two named authors. However, in this section, we will recapitulate some of the main definitions and results from \emph{op. cit.} focusing on a more intuitive presentation which sweeps some of the more complex coherence conditions under the rug. 

\begin{defn}
	Given a symmetric, monocolored operad $\scr{O}$, a $\scr{O}$-monoidal category consists of the following data: 
	\begin{enumerate}
		\item A category $\C$. 
		\item For every operation $z\in \scr{O}(n)$, an $n$-ary operation 
		\[
		\begin{tikzcd}
			\otimes_z:&[-3em] \C^{\times n}\arrow[r] & \C. 
		\end{tikzcd}
		\]
		\item For every $z\in \scr{O}(m)$, $x_1\in \scr{O}(n_1),\ldots,x_n\in \scr{O}(n_m)$, and permutation $\sigma$ of $m$, natural isomorphisms 
		\[
		\otimes_{z}\circ \left(\otimes_{x_{\sigma(1)}}\times\cdots\times \otimes_{x_{\sigma(n)}}\right)\cong \otimes_{(z\cdot \sigma)\circ(x_1,\ldots,x_n)}.
		\]
		We will denote these 2-isomorphisms by $\phi^{\sigma;z;x_1,\ldots,x_n}$.
	\end{enumerate} 
	Such that the $1$-ary operation associated to the operadic unit is the identity, and the natural isomorphisms of part (2) respect the associativity and unitality of composition. 
\end{defn}

\begin{rmk}
	This is a reformulation of the definition given in \cite{OMonGroth}. In that paper, all of the $n$-ary operations involved in the structure are packaged into a single functor 
	\[
	\begin{tikzcd}
		\scr{O}(n)\times \C^{\times n} \arrow[r] & \C, 
	\end{tikzcd}
	\]
	and likewise for the natural transformations. 
\end{rmk}

\begin{example}\
	\begin{enumerate}
		\item In the cases where $\scr{O}$ is the associative operad $\Assoc$ or the commutative operad $\Comm$, respectively, $\scr{O}$-monoidal categories are simply monoidal and symmetric monoidal categories, respectively, as shown in \cite[\S 3]{OMonGroth}.
		\item In the case of the trivial operad, $\scr{O}$-monoidal categories are simply categories. 
		\item $\on{QConv}$-monoidal categories have operations 
		\[
		\begin{tikzcd}
			\otimes_{\vec{\alpha}}: &[-3em] \scr{C}^{\otimes n} \arrow[r] & \scr{C} 
		\end{tikzcd}
		\]
		for every convex $n$-vector $\vec{\alpha}=(\alpha_1,\ldots,\alpha_n)$. There are two key kinds of structure isomorphisms: natural isomorphisms 
		\[
		\otimes_{(\alpha_1 \beta_1^1,\ldots \alpha_1\beta_{k_1}^{1},\alpha_2\beta_1^{2},\ldots, \alpha_n \beta_{k_n}^{n})}\cong \otimes_{\vec{\alpha}}\circ(\otimes_{\vec{\beta}^1}\times \cdots\times \otimes_{\vec{\beta}^n})
		\] 
		and permutation isomorphisms 
		\[
		\otimes_{\vec{\alpha}} \circ \sigma \cong \otimes_{\sigma\cdot \vec{\alpha}}.
		\]
	\end{enumerate}
\end{example}

\begin{defn}
	A \emph{lax $\scr{O}$-functor} between $\scr{O}$-monoidal categories $(\scr{C},\otimes,\phi)$ and $(\scr{D},\odot,\psi)$ consists of a functor $F:\scr{C}\to \scr{D}$, and, for every $z\in \scr{O}(n)$, a natural transformation 
	\[
	\begin{tikzcd}
		\xi^z:&[-3em]\odot_{z}\circ(F^{\times n})\arrow[r,Rightarrow] & F\circ\otimes_{z}. 
	\end{tikzcd}
	\]
	such that $\xi^u=\on{id}$ when $u$ is the operadic unit and the $\xi$'s commute with the $\phi$ and $\psi$. 
	
	An \emph{$\scr{O}$-monoidal transformation} between $\scr{O}$-monoidal functors $(F,\xi)$ and $(G,\zeta)$ is a 2-morphism $\mu:F\Rightarrow G$ which commutes with $\xi$ and $\zeta$. 
	
	With these definitions, the collection of $\scr{O}$-monoidal categories becomes a strict 2-category $\scr{O}\sf{Mon}$. We denote the hom-categories in this 2-category by $\on{Fun}^{\scr{O},\on{lax}}(-,-)$. 
\end{defn}

\begin{prop}
	A morphism of operads $f:\scr{P}\to \scr{O}$ induces a 2-functor 
	\[
	\begin{tikzcd}
		f^\ast: &[-3em] \scr{O}\sf{Mon}\arrow[r] & \scr{P}\sf{Mon} 
	\end{tikzcd}
	\]
	which sends an $\scr{O}$-monoidal category $(\scr{C},\otimes,\phi)$ to the $\scr{P}$-monoidal category $(\scr{C},f^\ast\otimes,f^\ast\phi)$, where $f^\ast\otimes_{z}=\otimes_{f(z)}$, and  $(f^\ast\phi)^{\sigma;z;x_1,\ldots,x_n}=\phi^{\sigma;f(z);f(x_1),\ldots,f(x_n)}$. 
\end{prop}

\begin{proof}
	This is \cite{OMonGroth}.
\end{proof}

In particular, since the commutative operad is terminal, this proposition tells us that for any operad $\scr{O}$, a symmetric monoidal structure on $\scr{C}$ induces a canonical $\scr{O}$-monoidal structure. We will call such $\scr{O}$-monoidal structures \emph{trivial}.

\begin{example}
	There is a non-trivial $\QConv$-monoidal structure $\star_{\vec{\alpha}}$ on the category $\CSet$ of convex sets given by defining 
	\[
	\begin{tikzcd}
		\star_{\vec{\alpha}}(X_1,\ldots,X_n) 
	\end{tikzcd}
	\]
	to be the subset of the join $X_1\star\cdots\star X_n$ on the elements of the form 
	\[
	\sum_{i=1}^n \alpha_i x_i
	\]
	for $x_i\in X_i$. The structure morphisms are given by those of the (trivial) join $\QConv$-structure. 
\end{example}

\begin{example}\
	\begin{enumerate}
		\item The identity functor $\CSet\to \CSet$ can be given the structure of a lax monoidal functor from $\pi^\ast\star$ to $\star_{\vec{\alpha}}$ , using the defining inclusions 
		\[
		\begin{tikzcd}
			\iota_{\vec{\alpha}}^{\vec{X}}:&[-3em]\star_{\vec{\alpha}}(X_1,\ldots,X_n) \arrow[r,hookrightarrow] & X_1\star \cdots \star X_n
		\end{tikzcd}
		\]
		\item The identity functor $\CSet\to \CSet$ can be given the structure of a lax monoidal functor from $\star_{\vec{\alpha}}$ to $\pi^\ast\times$ by using the morphisms 
		\[
		\begin{tikzcd}
			\ell_{\vec{\alpha}}^{\vec{X}}: &[-3em] X_1\times\cdots \times X_n \arrow[r] & \star_{\vec{\alpha}}(X_1,\ldots,X_n)
		\end{tikzcd}
		\]
		which send $(x_1,\ldots,x_n)$ to 
		\[
		\sum\alpha_ix_i.
		\]
		Notice that this formula yields well-defined maps 
		\[
		\begin{tikzcd}
			X_1\times\cdots \times X_n \arrow[r] & X_1\star\cdots\star X_n
		\end{tikzcd}
		\]
		it does \emph{not} define a lax symmetric monoidal functor, as the structure morphisms necessarily depend on the choice of $\vec{\alpha}$. 
	\end{enumerate}
\end{example}

In addition to these, we list a few lax symmetric monoidal functors, which will be of use in the sequel.

\begin{example} \label{example : QConv-mon}
		
		 The distributions monad $D_R: \sf{FinSet}\to \CSet$ is a lax $\QConv$-monoidal functor from $\pi^\ast\amalg$ to $\pi^\ast \otimes$ with structure maps 
		\[
		\begin{tikzcd}[row sep=0em]
			\xi_{\vec{\alpha}}: &[-3em] D_R(X_1)\times \cdots \times D_R(X_n) \arrow[r] & D(X_1\amalg\cdots \amalg X_n)\\
			& (p_1,\ldots,p_n) \arrow[r,mapsto] & {\displaystyle\sum_{i=1}^n} \alpha_i p_i 
		\end{tikzcd}
		\]
		where, in the latter sum, each $p_i$ is viewed as a distribution concentrated on the summand $X_i$. 
	
\end{example}

\subsection{The Grothendieck construction}

The starting point of our convex $\scr{O}$-monoidal Grothendieck construction is the main theorem of \cite{OMonGroth}, here we abusively by $\times$ the trivial $\scr{O}$-monoidal structure induced by the Cartesian symmetric monoidal structure.

\begin{defn}
	Let $(\scr{I},\otimes,\phi)$ be a small $\scr{O}$-monoidal category. An \emph{$\scr{O}$-monoidal fibration} over $\scr{I}$ consists of an $\scr{O}$-monoidal category $(\scr{C},\odot,\psi)$ and a \emph{strict} $\scr{O}$-monoidal functor $(\pi,\xi):\scr{C}\to\scr{I}$ such that $\pi$ is, additionally, a discrete fibration. A \emph{morphism} of such $\scr{O}$-monoidal fibrations is a commutative diagram 
	\[
	\begin{tikzcd}
		\scr{C} \arrow[rr,"f"]\arrow[dr,"\pi"'] & & \scr{D} \arrow[dl,"p"]\\
		& \scr{I} & 
	\end{tikzcd}
	\]
	such that $f$ is a \emph{strict} monoidal functor. We denote (somewhat abusively) by $\scr{O}\sf{Fib}_{\scr{I}}$ the category of $\scr{O}$-monoidal fibrations over $\scr{I}$. 
\end{defn}

\begin{thm}
	Let $(\scr{I},\otimes,\phi)$ be a small $\scr{O}$-monoidal category. Then the classical Grothendieck construction for discrete fibrations yields an equivalence of categories
	\[
	\begin{tikzcd}
		{\displaystyle\int_{\scr{I}}^{\scr{O}} 	}:&[-3em]\on{Fun}^{\scr{O},\on{lax}}\left((\scr{I},\otimes),(\Set,\times)\right) \arrow[r] & \scr{O}\sf{Fib}_{\scr{I}}. 
	\end{tikzcd}
	\]
\end{thm}

\begin{proof}
	This is \cite{OMonGroth}.
\end{proof}

We will extend this theorem into the realm of convex sets in two steps. Firstly, we consider a version of the Grothendieck construction whose input is an $\scr{O}$-lax functor $\scr{I}\to \Set$ together with an extension of the underlying functor to $\CSet$. We then consider $\scr{O}$-lax functors to $\CSet$ as a subcategory of this category, and identify its essential image.   

\begin{defn}
	Let $(\scr{I},\otimes,\phi)$ be a small $\scr{O}$-monoidal category. A \emph{fibrewise convex $\scr{O}$-monoidal fibration} over $\scr{I}$ consists of a $\scr{O}$-monoidal fibration $\pi:\scr{D}\to \scr{I}$ and a fibrewise convex structure on $\pi$. A \emph{morphism of fibrewise convex $\scr{O}$-monoidal fibrations} is simply a morphism of $\scr{O}$-monoidal fibrations which induces a convex map on each fibre over an object or morphism. We denote the category of fibrewise convex $\scr{O}$-monoidal fibrations over $\scr{I}$ by $\scr{O}\sf{FCFib}_{\scr{I}}$. 
\end{defn}
 
\begin{rmk}\label{rem:OFCFIB_hmtpypullback}
	Notice that $\scr{O}\sf{FCFib}_{\scr{I}}$ is simply the pullback of $\scr{O}\sf{Fib}_{\scr{I}}$ and $\sf{FCFib}_{\scr{I}}$ over the category $\sf{DFib}_{\scr{I}}$ of discrete fibrations over $\scr{I}$. Since the forgetful functor 
	\[
	\begin{tikzcd}
		\sf{FCFib}_{\scr{I}}\arrow[r] & \sf{DFib}_{\scr{I}}
	\end{tikzcd}
	\]
	is clearly an isofibration, this pullback is, in fact, a homotopy pullback in $\Cat$. 
\end{rmk}

\begin{defn}
	Given an $\scr{O}$-monoidal category $(\scr{I},\odot,\phi)$, we define a category $\on{HCFun}^{\scr{O},\on{lax}}((\scr{I},\odot),(\CSet,\otimes))$ to be the pullback 
	\[
	\Fun(\scr{I},\CSet)\times_{\Fun(\scr{I},\Set)} \Fun^{\scr{O},\on{lax}}((\scr{I},\odot),(\Set,\times)).
	\]
	Explicitly, the objects are lax $\scr{O}$-monoidal functors into $\Set$ together with a lift of the underlying functor to $\CSet$. One can equivalently think of this as a lax monoidal functor $(F,\xi)$ from $\scr{I}$ to $(\CSet,\times)$ in which the structure isomoprhisms $\xi$ are not required to be maps of convex sets. 
\end{defn}

\begin{lem}\label{lem:Half-convex_GC}
	For an $\scr{O}$-monoidal category $(\scr{I},\odot,\phi)$, the classical discrete Grothendieck construction induces an equivalence 
	\[
	\begin{tikzcd}
		{ \cint_{\scr{I}}^{\scr{O}}}: &[-3em]  \on{HCFun}^{\scr{O},\on{lax}}((\scr{I},\odot),(\CSet,\times)) \arrow[r] & \scr{O}\sf{FCFib}_{\scr{I}}.
	\end{tikzcd}
	\]
\end{lem}

\begin{proof}
	The functor is induced by the universal property of the pullback applied to the transformation of pullback diagrams given by the classical, na\"ive convex, and $\scr{O}$-monoidal Grothendieck constructions. Using the fact that $\Fun(\scr{I},\CSet)\to\Fun(\scr{I},\Set)$ is an isofibration in combination with Remark \ref{rem:OFCFIB_hmtpypullback}, we see that both pullbacks are homotopy pullbacks. Since the three component Grothendieck constructions are equivalences of categories, it follows that the induced functor is an equivalence, as desired.  
\end{proof}

We now note that we can identify $\Fun^{\scr{O},\on{lax}}((\scr{I},\odot),(\CSet,\otimes))$ with the full subcategory of\linebreak $\on{HCFun}^{\scr{O},\on{lax}}((\scr{I},\odot),(\CSet,\times))$ on those lax functors $(F,\xi)$ for which all of the structure maps 
\[
\begin{tikzcd}
	\xi^z_{i_1,\ldots,i_n}:&[-3em]\bigtimes_{z}(F(i_1),\ldots,F(i_n))\arrow[r,Rightarrow] & F\left(\otimes_{z}(i_1,\ldots,i_n)\right)
\end{tikzcd}
\]
are $n$-convex. It remains for us to identify the essential image of this functor in $\scr{O}\sf{FCFib}_{\scr{I}}$. 

\begin{defn}
	A fibrewise convex $\scr{O}$-monoidal fibration $\pi:(\scr{C},\boxtimes,\psi)\to (\scr{I},\odot,\phi)$ is called a \emph{$\scr{O}$-convex fibration} if, for any $z\in \scr{O}(n)$ and $i_1,\ldots,i_n\in \scr{I}$, the induced map 
	\[
	\begin{tikzcd}
		\pi^{-1}(i_1)\times \cdots \times \pi^{-1}(i_n) \arrow[r,hookrightarrow] & \scr{C}^n \arrow[r,"\odot_{z}"] & \C 
	\end{tikzcd}
	\]
	is $n$-convex. We will denote the full subcategory of $\scr{O}\sf{FCFib}_{\scr{I}}$ on the $\scr{O}$-convex fibrations by $\scr{O}\sf{CFib}_{\scr{I}}$. 
\end{defn}

\begin{thm}
	For a $\scr{O}$-monoidal category $(\scr{I},\odot,\phi)$, the classical Grothendieck construction induces an equivalence of categories 
	\[
	\begin{tikzcd}
		\cint_{\scr{I}}^\scr{O}:&[-3em] \Fun^{\scr{O},\on{lax}}((\scr{I},\odot),(\CSet,\otimes))\arrow[r] & \scr{O}\sf{CFib}_{\scr{I}}.
	\end{tikzcd}
	\]
\end{thm}

\begin{proof}
	It is immediate from the definitions that the essential image of $\Fun^{\scr{O},\on{lax}}((\scr{I},\odot),(\CSet,\otimes))$ under the functor of Lemma \ref{lem:Half-convex_GC} is $\scr{O}\sf{CFib}_{\scr{I}}$, completing the proof. 
\end{proof}

\section{Examples and applications}\label{sec:EGs}

\subsection{Finite probability spaces and entropy}\label{subsec:Entropy}

In \cite{BaezFritzLeinster}, the authors prove a theorem which formulates \emph{information loss entropy} in categorical jargon. We briefly recall this result and propose a more synthetic reformulation using the technology of convex objects developed in the earlier section. 

Our category of interest is $\Prob$. Recall that the objects of this category are pairs $(X,p)$ where $X \in \Fin$ is a finite set and $p : X \to \RR_{\geq 0}$ is a probability measure, which just means $\sum_{x \in X} p(x) = 1$. A morphism $f : (X, p) \to (Y, q)$ between such pairs is a function $f: X \to Y$ such that for all $y \in Y$ we have $q(y) = \sum_{x \in f^{-1}(y)} p(x)$. Such functions are called \emph{measure preserving}.

    Let $(X, p) \in \Prob$ be a probability measure on a finite set $X$. The \emph{Shannon entropy} associated to this object is defined to be 
    $$H(X, p) = - \sum_{x \in X} p(x) ln(p(x))$$
    For a morphism $f : (X,p) \to (Y,q)$, the closely related quantity
    $$F(f) = H(Y, q) - H(X ,p)$$
    is of interest as it measures \emph{information loss} by $f$. In categorical jargon, $F$ is a mapping which assigns a non-negative real number to a morphism in $\Prob$. It turns out that, up to scalar constant, $F$ is  uniquely determined by some properties it satisfies.

    \begin{thm}[\cite{BaezFritzLeinster}, Theorem 2] \label{thm: entropy}
         Let $F : \Mor(\Prob) \to \RR_{\geq 0}$ be a function which satisfies the following conditions:
    \begin{itemize}
        \item [(i)] $F$ respects \emph{composition}, i.e.  whenever $f$ and $g$ are composable we have
        $$F(f \circ g) = F(f) + F(g)$$

        \item[(ii)] $F$ respects \emph{convexity}, i.e. for all $\lambda \in [0, 1]$ and $f, g$ we have
        $$F(\lambda f + (1-\lambda) g) = \lambda F(f) + (1-\lambda) F(g)$$
        Here, the left hand-side means the following. For $f : (X_1, p_1) \to (Y_1, q_1)$ and $g : (X_2, p_2) \to (Y_2, q_2)$, the convex combination $\lambda f + (1- \lambda)g$ is the unique morphism $(X_1 \coprod X_2, \lambda p_1 + (1-\lambda)p_2) \to (Y_1 \coprod Y_2, \lambda q_1 + (1-\lambda)q_2)$. 

        \item[(iii)] $F$ is \emph{continuous} in the following sense. We say that a sequence of morphisms $f_n : (X_n , p_n) \to (Y_n, q_n)$ converges to $f : (X, p) \to (Y, q)$ in case $X_n = X$, $Y_n = Y$ and $f_n = f$ for large $n$ and moreover $p_n(x)$ and $q_n(y)$ converge pointwise to $p(x)$ and $q(y)$ for all $x \in X, y \in Y$. Continuity of $F$ means that in such case, $F(f_n)$ converges to $F(f)$.
    \end{itemize}
    Then, $F$ is of the form
    $$F(f) = c(H(Y, q) - H(X ,p))$$
     for a unique scalar $c$, for all morphisms $f : (X,p) \to (Y, q)$ in $\Prob$.
    \end{thm}

    Leaving conditions (i) and (iii) aside for the moment (we will return to them in a bit), condition (ii) in the above result seems ad-hoc. We ask the following natural question.
    \begin{question}
        Where did the morphisms in $\Prob$ acquire a convex structure from?
    \end{question}
    In light of our machinery, we may propose the following answer.
    \begin{answer}
        $\Prob$ inherits a convex structure by virtue of being a convex Grothendieck construction. 
    \end{answer}
    We briefly explain. Consider the functor $\sf{Dist} : \Fin \to \Set$ be the functor which assigns to a finite set $X$ the set of probability distributions $D(X)$. Then, $\Prob$ is the Grothendieck construction of this functor,
    $$\Prob = \int_\Fin \sf{Dist}$$
    Extra structure on $\Prob$ reflects extra structure on the base functor $\sf{Dist}$. To begin with, all the sets $D(X)$, $X \in \Fin$, are convex sets and all induced morphisms are convex as well. Hence, we may promote $\sf{Dist}$ to a functor
    $$\sf{Dist} : \Fin \to \CSet$$
    However, this much structure would only equip a Grothendieck construction with e fibre-wise convex structure in terms of the structure forgetful functor $\pi : \Prob \to \Fin$ (as discussed in the previous section), not enough to explain condition (ii) above.

    For more structure, notice that for any two finite sets $X$ and $Y$ we have $D(X \amalg Y) \cong D(X) * D(Y)$. This tells us that $\sf{Dist}$ (weakly) preserves coproducts, or more generally that $\sf{Dist}$ is a monoidal functor. This would allow us to apply the monoidal Grothendieck construction (\cite{MoellerVasilakopoulou}), but this would only equip $\Prob$ with a monoidal structure. 

    The convex structure which appear in condition (ii) above is really encoded when we regard $\sf{Dist}$ as a lax $\QConv$-monoidal functor
    $$(\sf{Dist, \xi}) : (\Fin, \amalg) \to (\CSet, \otimes)$$
    We regard $(\Fin, \coprod)$ as a trivial $\QConv$-monoidal category, while $(\CSet, \otimes)$ has the $\QConv$-monoidal structure detailed in Example \ref{example : QConv-mon}. For each $\lambda \in [0,1]$, we define the structure map for the lax structure to be
    \[
    \begin{tikzcd}
        \xi_\lambda : &[-3em] D(X) \times D(Y) \arrow[r] & D(X) * D(Y) \\[-2em]
        & (p,q) \arrow[r, mapsto] & \lambda p + (1-\lambda)q
    \end{tikzcd}
    \]
    $\xi_\lambda$ is a convex map, and hence biconvex, so that we may regard it as a morphism from the tensor product $D(X) \otimes D(Y)$. In conclusion, the relevant convex structure is obtained when we regard
    $$\Prob = \cint_\Fin \sf{Dist}$$
    equipped with the $\QConv$-monoidal structure. 

    Besides convexity, the category $\Prob$ has some topological structure. This is also a consequence of $\Prob$ being a Grothendieck construction, since the topological features are due to the fact that for all $X \in \Fin$, the set $D(X) \subseteq \RR^{|X|}$ has a subspace topology (in fact, $D(X)$ is a convex topological space). This being said, the known notions of category enriched in topological spaces or topological category do not encompass $\Prob$ properly. 

    \begin{defn}
        A category $\C$ is said to be a \emph{fibre-wise topological category} with respect to a functor $\pi : \C \to \J$ if:
        \begin{itemize}
            \item For each object $i \in \J$, the set $\pi^{-1}(i)$ has a topology.
            \item For each morphism $f \in \Mor(\J)$, the preimage $\pi^{-1}(f)$ has a topology.
            \item The source and target functions, $s : \pi^{-1}(f) \to \pi^{-1}(s(f))$ and $t : \pi^{-1}(f) \to \pi^{-1}(t(f))$, are continuous for all morphisms $f \in \Mor(\J)$.
        \end{itemize}

        A morphism of fibre-wise topological categories $\pi : \C \to \J$ and $\theta: \D \to \I$ is a commutative square of functors
        \[\begin{tikzcd}
	\C & \D \\
	\J & \I
	\arrow["\pi"', from=1-1, to=2-1]
	\arrow["\theta", from=1-2, to=2-2]
	\arrow["F", from=1-1, to=1-2]
	\arrow["T"', from=2-1, to=2-2]
\end{tikzcd}\]
        such that 
        \begin{itemize}
            \item For all objects $i \in \J$, the induced map $F|_{\pi^{-1}(i)} : \pi^{-1}(i) \to \theta^{-1}(T(i))$ is continuous.
            \item For all morphisms $f \in \Mor(\J)$, the induced map $F|_{\pi^{-1}(f)} : \pi^{-1}(f) \to \theta^{-1}(T(f))$ is continuous.
        \end{itemize}
        We abuse terminology and just say $F : \C \to \D$ is \emph{continuous} and leave the rest of the structure understood in context.
    \end{defn}

    \begin{example}
        If a category $\C$ is an internal category in topological spaces then it is fibre-wise topological with respect to the terminal functor $\C \to *$. In particular, the topological monoid $(\RR_{\geq 0} , +)$ may be regarded as a one-object internal category in spaces $B\RR_{\geq 0}$ and hence as a fibre-wise topological category.
    \end{example}

    \begin{example}
        Let $\J$ be a small category and
        \[
        \begin{tikzcd}
            G : &[-3em] \J \arrow[r] & \sf{Top}
        \end{tikzcd}
        \]
        be a functor into a category of nice topological spaces. The Grothendieck construction $\int_\J G$ has the structure of a fibre-wise topological category with respect to the associated discrete fibration $\pi : \int_\J G \to \J$. 

        This is not difficult to check. For an object $i \in \J$, we have $\pi^{-1}(i) = G(i)$ and the latter is a topological space. For a morphism $f : i \to j$ in $\J$, we have $\pi^{-1}(f) = \{ (x, y) \ | \ x \in G(i), y \in G(j), (Gf)(x) = y \} \subset G(i) \times G(j)$, so we may equip $\pi^{-1}(f)$ with the subspace topology. The source and target functions fit in commutative triangles
        \[\begin{tikzcd}[column sep = small]
	{\pi^{-1}(f)} && {G(i)} \\
	& {G(i) \times G(j)}
	\arrow[hook', from=1-1, to=2-2]
	\arrow["pr"', two heads, from=2-2, to=1-3]
	\arrow["s", from=1-1, to=1-3]
\end{tikzcd} \ \ 
, \ \ 
\begin{tikzcd}[column sep = small]
	{\pi^{-1}(f)} && {G(j)} \\
	& {G(i) \times G(j)}
	\arrow[hook', from=1-1, to=2-2]
	\arrow["pr"', two heads, from=2-2, to=1-3]
	\arrow["t", from=1-1, to=1-3]
\end{tikzcd}
\]
where they are exhibited as being factored as an inclusion, which is assumed to be continuous by definition, and a projection. Hence, these maps are continuous. 
    \end{example}

    \begin{example}
        In particular, $\Prob$ is a fibre-wise topological category with respect to the forgetful functor $\Prob \to \Fin$ by virtue of being the Grothendieck construction of the $\sf{Top}$-valued functor $\sf{Dist}$. 
    \end{example}

    The next thing we would like to remark in the very setup of the categorical characterisation of information loss entropy. The result is categorical, in the sense that it is a statement about morphisms of a category and condition (i) is stated in terms of composition. Nonetheless, it strange, at least from the categorical angle, to study functions from the collection of morphisms in a category. This may be addressed by considering a functor $F : \Prob \to B\RR_{\geq 0}$. 

    We can summarise our observations in the form of a following fully-functorial reformulation of Theorem \ref{thm: entropy}.

    \begin{thm}
        Any continuous $\QConv$-monoidal functor $F \in \Fun^{\QConv, \on{cts}} (\Prob, B\RR_{\geq 0})$ is of the form
        $$F(f) = c(H(X,p) - H(Y, q))$$
        for some constant $c$,
        for all morphisms $f : (X,p) \to (Y,q)$ in $\Prob$.
    \end{thm}

    \begin{rmk}
        We believe the statement of the theorem above can be structurally pushed a step further by saying that there is an isomorphism of convex topological monoids 
        $$\Fun^{\QConv, \on{cts}} (\Prob, B\RR_{\geq 0}) \cong \RR_{\geq 0}.$$
        
    \end{rmk}

\subsection{The convex monoid of twisted distributions}

In \cite[\S 4]{KharoofSimplicial}, the authors define notions of weak and strong invertibility for what they call \emph{convex monoids}. Unwinding their definitions, a convex monoid is a convex set, equipped with the structure of a monoid, such that the multiplication map is biconvex, i.e., a monoid object in $(\CSet,\otimes)$. These convex monoids, and the associated notions of weak and strong invertibility are of substantial interest, since they allow one to characterize quantum contextuality within the framework  of simplicial distributions from \cite{OkaySimplicial}. 

\begin{defn}
	The \emph{simplex category} $\Delta$ has objects $[n]=\{0,1,\ldots,n\}$ for $n\geq 0$, and morphisms given by weakly monotone maps. A \emph{simplicial set} is a functor
	\[
	\begin{tikzcd}
		X:&[-3em] \Delta^\op\arrow[r] & \Set
	\end{tikzcd}
	\]
\end{defn}

The theory of simplicial distributions presented in \cite{OkaySimplicial,KharoofSimplicial} models quantum contextuality by defining a simplicial set $X$ of measurements, and a simplicial set $Y$ of outcomes. In this context, an $n$-simplex of $X$ can be thought of as a \emph{measurement context} --- an $n$-tuple of measurements which can be performed simultaneously (e.g., commuting hermitian operators). The outcomes of quantum measurements are then a collection of probability distributions on contexts which are compatible with the simplicial identities. This compatibility encodes the so-called \emph{non-signaling conditions} of quantum mechanics. 

Formally, given simplicial sets $X$ and $Y$ a simplicial distribution from $X$ to $Y$ is a map of simplicial sets 
\[
\begin{tikzcd}
	p:&[-3em] X \arrow[r] & D(Y)
\end{tikzcd}
\]
or, equivalently, a morphism in the Kleisli category of the monad $D$ acting levelwise on the category $\Set_\Delta$. 

In recent work \cite{TwistDist} of the second and third-named authors, this framework of simplical distributions is extended to the setting of simplicial principal bundles. This extension captures the features of a number of examples arising from quantum computation. 

\begin{defn}
	A simplicial Abelian group is a functor 
	\[
	\begin{tikzcd}
		K: &[-3em] \Delta^\op \arrow[r] & \sf{Ab} 
	\end{tikzcd}
	\]
	where $\sf{Ab}$ denotes the category of Abelian groups. A \emph{principal $K$-bundle} consists of a left $K$-action on a simplicial set $E$ and a map of simplicial sets $\pi:E\to X$ such that the action on $E$ is free, and $\pi$ is the quotient map. A \emph{simplicial distribution} on a principal $K$-bundle $\pi:E\to X$ is a map of simplicial sets $p$ fitting into a commutative diagram 
	\[
	\begin{tikzcd}
		 & D(E)\arrow[d,"D(\pi)"]\\
		X\arrow[ur,"p"]\arrow[r,"\delta"'] & D(X)
	\end{tikzcd}
	\]
	where $\delta$ sends each simplex to the delta distribution on that simplex. We write $\sDist(\pi)$ for the convex set of simplicial distributions on $\pi$.  
\end{defn}  

The convex monoids studied in \cite{KharoofSimplicial} have a more categorical avatar in this setting, which arises via a monoidal Grothendieck construction. Writing $\sf{Bun}_K(X)$ for the groupoid of principal $K$-bundles over $X$, we can equip $\sf{Bun}_K(X)$ with a monoidal structure $-\otimes_K-$ defined by taking the quotient of $E\times_X F$ by the relation 
\[
(x,y)\sim (kx,k^{-1}y)
\]
for $k\in K$. By \cite[Theorem 2.30]{TwistDist}, this is a symmetric monoidal structure. 

Moreover, for any simplicial sets $X$ and $Y$, we can define maps 
\[
\begin{tikzcd}[row sep=0em]
	m_{X,Y}:&[-3em] D(X)\times D(Y) \arrow[r] & D(X\times Y)\\
	 & (p,q) \arrow[r,mapsto] & ((x,y)\mapsto p(x)q(x))
\end{tikzcd}
\]  
For a pair of principal bundles $\pi_E:E\to X$ and $\pi_F:F\to X$, the map $m_{X,Y}$ induces a map 
\[
\begin{tikzcd}
	\mu_{E,F}: &[-3em] \sDist(\pi_E)\times \sDist(\pi_F) \arrow[r] & \sDist(\pi_{E\otimes_K F})
\end{tikzcd}
\]
since these maps are biconvex, they may be viewed as convex maps  
\[
\begin{tikzcd}
	\mu_{E,F}: &[-3em] \sDist(\pi_E)\otimes \sDist(\pi_F) \arrow[r] & \sDist(\pi_{E\otimes_K F})
\end{tikzcd}
\]
out of the convex tensor product. We then have 

\begin{prop}\textup{\cite[Prop. 3.11]{TwistDist}}
	The maps $\mu_{E,F}$ are the structure maps of a lax symmetric monoidal functor 
	\[
	\begin{tikzcd}
		\sDist: &[-3em] (\sf{Bun}_K(X),\otimes_K) \arrow[r] & (\CSet,\otimes).
	\end{tikzcd}
	\]
\end{prop}

From the convex $\O$-monoidal Grothendieck construction, we them obtain 

\begin{cor}
	The category $\cint_{\sf{Bun}_K(X)} \sDist$ is a fibrewise convex monoidal category over $\sf{Bun}_K(X)$.
\end{cor}

Of particular computation use in the study of simplicial principal bundles, one may define \emph{twisting functions}, collections of maps $\{\eta_n:X_n\to K_{n-1}\}$ satisfying some compatibilities with simplicial maps. From a twisting function $\eta$, one may define a \emph{twisted product} $\pi_\eta: K\times_\eta X\to X$ which is a simplicial principal bundle (see \cite{May67} or \cite{TwistDist} for details). Twisting functions form a commutative monoid $\on{Twist}_K(X)$ under addition, and, viewing this as a discrete symmetric monoidal category, we have the following relation with the monoidal structure on $\sf{Bun}_K(X)$. 

\begin{prop}
	The functor 
	\[
	\begin{tikzcd}[row sep=0em ]
		\on{Twist}_K(X) \arrow[r] & \sf{Bun}_K(X) \\
		\eta \arrow[r,mapsto] & K\times_\eta X 
	\end{tikzcd}
	\]
	is symmetric monoidal and essentially surjective. 
\end{prop} 

As a result, the convex Grothendieck construction allows us to conclude that 
\[
\coprod_{\on{Twist}_K(X)} \sDist(\pi_\eta)
\]
is a monoid over $\on{Twist}_K(X)$, and that the multiplication is fibrewise biconvex. This is the avatar, in the bundle setting, of the convex monoid of simplicial distributions studied in \cite{KharoofSimplicial}, and, indeed, the fibre over the trivial twisting function zero retrieved the convex monoid studied in op. cit.

	\printbibliography

\end{document}